\documentclass[microtype]{gtpart}     %

\usepackage{pinlabel}  
\usepackage{graphicx}  
\usepackage[font=small]{caption}

\usepackage{mathrsfs}
\usepackage{bbm}
\usepackage{tikz}
\usepackage{tikz-cd}
\usepackage{booktabs}
\usepackage{comment}

\title{The Heisenberg Plane}

\author{Steve J. Trettel}
\givenname{Steve}
\surname{Trettel}
\address{}
\email{trettel@stanford.edu}
\urladdr{www.stevejtrettel.site}

\title[The Heisenberg Plane]{The Heisenberg Plane}
\newtheorem{theorem}{Theorem}[section]    
\newtheorem{proposition}[theorem]{Proposition}
\newtheorem{corollary}[theorem]{Corollary}

\theoremstyle{definition}
\newtheorem{definition}[theorem]{Definition}    
     
\newtheorem{observation}{Observation}        
\newtheorem*{claim}{Claim}
\renewcommand{\D}{\mathbb{D}}

\newcommand{\A}{\mathbb{A}}
\renewcommand{\H}{\mathbb{H}}
\newcommand{\E}{\mathbb{E}}

\renewcommand{\S}{\mathbb{S}}
\newcommand{\RP}{\mathbb{R}\mathsf{P}}

\newcommand{\Hs}{\mathbb{H}\mathbbm{s}}

\newcommand{\ep}{\varepsilon}

\newcommand{\mat}[1]{\begin{matrix} #1 \end{matrix}}

\newcommand{\pmat}[1]{\begin{pmatrix} #1 \end{pmatrix}}

\newcommand{\smat}[1]{\left( \begin{smallmatrix} #1 \end{smallmatrix}\right )}

\newcommand{\Isom}{\mathsf{Isom}}
\newcommand{\Euc}{\mathsf{Euc}}

\newcommand{\Heis}{\mathsf{Heis}}

\newcommand{\diag}{\mathsf{diag}}

\newcommand{\GL}{\mathsf{GL}}
\newcommand{\PGL}{\mathsf{PGL}}

\newcommand{\SL}{\mathsf{SL}}

\renewcommand{\O}{\mathsf{O}}
\newcommand{\PO}{\mathsf{PO}}
\newcommand{\SO}{\mathsf{SO}}

\newcommand{\heis}{\mathfrak{heis}}

\newcommand{\inject}{\hookrightarrow}
\newcommand{\surject}{\twoheadrightarrow}

\DeclareMathOperator{\Hom}{Hom}
\newcommand{\dev}{\mathsf{dev}}
\newcommand{\hol}{\mathsf{hol}}
\newcommand{\Area}{\mathsf{Area}}
\newcommand{\inv}{^{-1}}

\begin{document}

\begin{abstract}   
The geometry of the Heisenberg group acting on the plane arises naturally in geometric topology as a degeneration of the familiar spaces $\S^2,\H^2$ and $\E^2$ via \emph{conjugacy limit} as defined by Cooper, Danciger, and Wienhard.
This paper considers the deformation and regeneration of Heisenberg structures on orbifolds, adding a carefully worked low-dimensional example to the existing literature on geometric transitions.
In particular, the closed orbifolds admitting Heisenberg structures are classified, and their deformation spaces are computed.
Considering the regeneration problem, which Heisenberg tori arise as rescaled limits of collapsing paths of constant curvature cone tori is completely determined in the case of a single cone point.
\end{abstract}

\maketitle


\section{Introduction}

Heisenberg geometry is a geometry on the plane given by all translations together with shears parallel to a fixed line.
Viewing this fixed line as `space' and any line intersecting it transversely as `time,' this is the geometry of $1+1$ dimensional Galilean relativity.

\begin{definition}
Heisenberg geometry is the $(G,X)$ geometry $\Hs^2:=(\Heis,\A^2)$ where
$$\Heis=\left\{\pmat{\pm1 & a & c\\0 &\pm1& b\\0&0&1}\;\Bigg | \; a,b,c\in\R\right\} \hspace{0.2cm}\textrm{and}
\hspace{0.2cm}
\A^2=\left\{ [x:y:1]\in\RP^2\mid x,y,\in\R\right\}.$$
The identity component $\Heis_0<\Heis$ is the real Heisenberg group, and the index 2 subgroup of orientation-preserving transformations is denoted $\Heis_+$.
\end{definition}

The Heisenberg plane represents a particularly simple example of a non-Riemannian degeneration of Riemannian symmetric spaces via conjugacy limit, as studied by Cooper, Danciger, Weinhard, Fillastre, Seppi and others \cite{Danciger11, CooperDW14,FillastreS16}.
The semi-Riemannian geometries with automorphism groups $\O(p,q)$ and their degenerations form a poset\footnote{The fact that spheres of increasing radius limit to their tangent plane can be used to produce a degeneration of spherical geometry to Euclidean showing that  $\E^2\prec\S^2$ for example.} with a minimum element in each dimension  \cite{CooperDW14}.
This `most degenerate' geometry has the property that no nontrivial orthogonal group of any dimension appears as a subgroup of its automorphisms, and in dimension two is the Heisenberg plane.

\begin{figure}
\includegraphics[width=\textwidth]{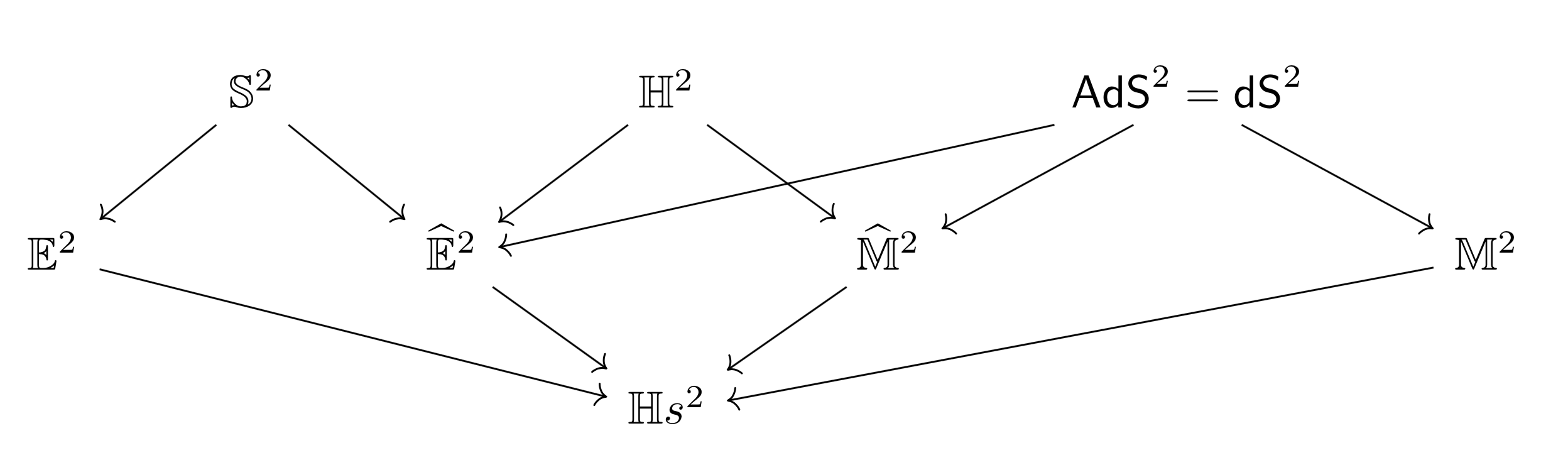}
\caption{The poset of subgeometries of $\RP^2$ with automorphism groups $\PO(p,q)$ (spherical, hyperbolic and (anti)-de Sitter space) and their degenerations (Adapted from \cite{CooperDW14}).
The first degenerations are geometries of Euclidean and Minkowski space together with their contragredient dual representations ($\widehat{\mathbb{M}}^2$ is the \emph{Half-Pipe} geometry of \cite{Danciger11}).
The Heisenberg plane is a degeneration of all of these.}
\end{figure}

This paper attempts to provide a detailed exploration of Heisenberg geometry, to add to the literature describing explicit geometric transitions.
We pay pay particular attention to aspects of interest to geometric topology; namely classifying Heisenberg orbifolds, calculating deformation their spaces and constructing regenerations of Heisenberg structures into familiar geometries.
In order to lower the prerequisites, when some result for the Heisenberg plane is a consequence of more general geometric theorems we mention this, but attempt to also provide self-contained proofs when possible and succinct.

\subsection{Heisenberg Orbifolds}

The first main result concerns the moduli problem for Heisenberg orbifolds.  
As a subgeometry of the affine plane, all Heisenberg orbifolds are finitely covered by a torus, so computing the deformation space $\mathcal{D}_{\Hs^2}(T^2)$ is the natural starting point.
Geometric structures on tori generalize elliptic curves (the conformal structures), especially in the presence of a compatible group operation.
As in the complete affine case studied by Goldman and Baues \cite{BauesGoldman05}, each Heisenberg torus admits a group structure with Heisenberg maps realizing the group operation, which we explicitly describe.
As a first step to determining these structures we compute the representation variety of potential holonomies.

\begin{theorem}
The representation variety $\Hom(\Z^2,\Heis_0)$ is isomorphic as a real algebraic variety to the product $V\times\R^2$, where $V$ is the 3-dimensional variety $V=\{(x,y,z,w)\in\R^4\mid xy=zw\}$.  
Topologically, this is homeomorphic to the product of a plane with the cone on a torus.
\end{theorem}

The Heisenberg plane admits no invariant Riemannian metric, so the possibility of incomplete structures must be taken seriously.
In contrast to the affine case \cite{Baues14,Nagano74} however, a geometric argument shows all Heisenberg structures are complete, and the deformation space $\mathcal{D}_{\Hs^2}(T^2)$ of tori identifies with the conjugacy classes of faithful representations $\Z^2\to\Heis_+$ acting properly discontinuously on $\R^2$.
The projection onto conjugacy classes admits a section allowing us to select a preferred holonomy (and construct the corresponding developing map) for each point in deformation space.

\begin{theorem}
All Heisenberg tori are complete, and the projection onto holonomy $\mathcal{D}_{\Hs^2}(T^2)\to \Heis(\Z^2,\Heis_+)/\Heis_+$ is an embedding.
The deformation space identifies with the classes of faithful representations acting properly discontinously, and is homeomorphic to $\R^3\times\S^1$.
\end{theorem}

An explicit description of the deformation space of tori greatly simplifies the calculation of the remaining deformation spaces, which is relegated to an appendix.
As all structures are complete, the problem of determining Heisenberg structures on an orbifold $\mathcal{O}$ finitely covered by $T^2$ is equivalent to the following algebraic \emph{extension problem}: when does a representation $\rho\colon \pi_1(T)^2\to\Heis$ extend to a representation of $\pi_1(\mathcal{O})>\pi_1(T^2)$?

\begin{theorem}
There are nine closed Heisenberg orbifolds, namely the quotients of the torus with at most order two cone points and right angled reflector corners.  All Heisenberg orbifolds are complete, and the holonomy map $\hol\colon\mathcal{D}_{\Hs^2}(\mathcal{O})\to\Hom(\pi_1(\mathcal{O}),\Heis)/\Heis_+$ is an embedding.

\begin{table}[h]
\centering
\begin{tabular}{ccc}
\toprule
$\mathcal{O}$ &\hspace{1cm}& $\mathcal{D}_{\Hs^2}(\mathcal{O})$ \\
\midrule
$\S^1\times\S^1$ &\hspace{1cm}& $\R^3\times\S^1$ 
\\
$\S^1\widetilde{\times}\S^1$, $\S^1\times I$, $\S^1\widetilde{\times}I$ &\hspace{1cm}& $\R^3\sqcup\R^2$  
\\
$\S^2(2,2,2,2)$ &\hspace{1cm}& $\R^2\times\S^1$ 
\\
$\D^2(2,2;\varnothing),\; \D^2(\varnothing; 2,2,2,2)$ &\hspace{1cm}& $\R^2\sqcup\R^2$
\\
$\RP^2(2,2)\;\D^2(2;2,2)$ &\hspace{1cm}& $\R^2\sqcup\R^2$
\\
\bottomrule
\end{tabular}
\caption{The Heisenberg orbifolds and the homeomorphism type of their deformation spaces. }
\end{table}

\end{theorem}

\vspace{-0.5cm}

\subsection{Regenerating Heisenberg Tori}

The second main result of this paper concerns the regeneration of Heisenberg structures to constant curvature ones, adding a detailed example to the collection of regenerations studied by Danciger, Gu\'eritaud, Kassel, Hodgson, Leitner, Porti, and others: for a selection of relevant works see \cite{Danciger11Ideal,DanGK16,Leitner16-Cusp, Porti98,PortiHS01}.
Understanding the behavior of geometric structures along a transition is in general difficult, as one cannot directly use techniques from either geometry involved.
Suitably constructing degenerations of $\S^2,\E^2$ and $\H^2$ to the Heisenberg plane within the projective plane allows us to use constructions in projective geometry to bridge the gap and overcome the additional difficulty posed by lack of invariant metric on $\Hs^2$.

As the Heisenberg plane is a common degeneration of the familiar constant curvature geometries, focusing on tori we ask when a given Heisenberg torus is the rescaled limit of a sequence of constant curvature conemanifold structures.
Restricting to structures with at most one cone point, this has a clean resolution illustrating a stark dichotomy between two `flavors' of Heisenberg tori;
\emph{translation tori} with holonomy images intersecting $\Heis_0$ only in translations, and \emph{shear tori} have holonomy image containing a nontrivial shear.
 
\begin{theorem}
Let $T$ be a Heisenberg torus, and $\mathbb{X}\in\{\S^2,\E^2,\H^2\}$.  Then if $\mathbb{X}_t$ is a sequence of conjugate models of $\mathbb{X}$ limiting to the Heisenberg plane within $\RP^2$, there is a sequence of $\mathbb{X}_t$-cone tori $T_t$ with a single cone point limiting to $T$ if and only if $T$ is a translation torus.
\end{theorem}

A constructive argument for the `if' direction builds for each translation torus $\R^2/(\Z \vec{v}\oplus \Z\vec{w})$ a fundamental domain $Q\subset \R^2$ and a sequence of collapsing $\mathbb{X}$ cone tori such that under rescaling $\mathbb{X}$ degenerates to $\Hs^2$ and the rescaled fundamental domains converge to $Q$.
This construction is analogous to the regeneration of Euclidean tori as hyperbolic cone tori.
The `only if' direction follows from a geometric characterization of Heisenberg tori, relating shears in the image of the holonomy homomorphism to the distribution of simple geodesics on the surface.

\begin{theorem}
A Heisenberg orbifold $\mathcal{O}$ has a nontrivial shear in its holonomy if and only if all simple geodesics on $\mathcal{O}$ are parallel.
\end{theorem}

This provides a clear obstruction to regenerating shear tori.  
Any two simple geodesics on a shear torus are disjoint, but constant curvature cone tori with a single cone point have geodesic representatives of each homotopy class.
In particular, any generating set for $H_1(T^2)$ can be pulled tight to give intersecting simple geodesics.
An argument in projective geometry shows that any limit of $\mathbb{X}\in\{\S^2,\H^2,\E^2\}$-cone tori as $\RP^2$ structures inherits a collection of intersecting simple geodesics, finishing the proof.

\section*{Acknowledgements}
I am immensely grateful to Darren Long, my advisor, and Daryl Cooper for many helpful discussions.  I have learned a great deal from their helpful suggestions, and appreciate their  patience as I worked on this manuscript.  Thanks also to Gordon Kirby for listening to me develop these ideas, and much thanks to the anonymous reviewer whose detailed and constructive feedback substantially improved this paper.

\section{Background}

Following is a list of terminology and notations used throughout the paper for quick reference. 
We denote by $\Heis_0$ the real Heisenberg group of upper triangular unipotent $3\times 3$ matrices, $\Heis=(\Z_2)^2\rtimes\Heis_0$ the group generated by this together with reflections $\diag(\pm 1, \pm 1, 1)$.
 $\Heis_+$ is the index two orientation preserving subgroup of $\Heis$, and $\mathsf{Tr}$ the subgroup acting by translations on the plane.
  The Lie algebra $\heis$ consists of the strictly upper triangular $3\times 3$ matrices, and provides useful coordinates for the representation varieties.
For ease of inline typesetting we will often denote the element $\smat{0& x& z\\0&0&y\\0&0&0}\in\heis$ by the shorthand notation $\smat{x&z\\&y}$.

We denote a closed two dimensional orbifold $\mathcal{O}$ with underlying topological space $X$ by $X(\vec{c})$ if $\mathcal{O}$ has cone points of order $\vec{c}=(c_1,\ldots, c_m)$ and by $X(\vec{c};\vec{r})$ if in addition $\partial X\neq \varnothing$ and $\mathcal{O}$ has corner reflectors of order $\vec{r}=(r_1,\ldots, r_n)$.

The algebraic variety cut out by $f\in \R[x_1,\ldots x_n]$ is denoted $V(f)$.
A finite presentation for a group $\Gamma=\langle s_1,\ldots s_n|r_1,\ldots r_m\rangle$ gives an injection $\mathsf{ev}\colon\Hom(\Gamma, \Heis)\inject\Heis^n$ by evaluation on generators; $\mathsf{ev}(\rho)=(\rho(s_1),\ldots \rho(s_n))$.
The image is an algebraic variety cut out by the polynomials $\{r_i-I\}$. 
Pulling this structure back via the evaluation map equips the set of homomorphisms with a variety structure, which is independent of original choice of presentation (see for example \cite{Goldman88}).

\subsection{Klein Geometry}

A \emph{geometry} in the sense of Klein is a pair $(G,X)$ consisting of a Lie group $G$ 
acting analytically and transitively on a smooth manifold $X$.
Examples of Kleinian geometries abound in geometric topology,
from spherical geometry as the sphere with an $\SO(3)$ action to the hyperbolic plane as a disk in $\C$ together with the M\"obius transformations preserving it, and even non-Riemannian examples such as  projective space $(\SL(n+1;\R),\RP^n)$.
Consult \cite{Thurston80,Goldman10,CooperDW14} for additional reference and examples.

For convenience we often work with pointed geometries $(G,(X,x))$ selecting a particular point stabilizer $G_x=\mathsf{stab}_G(x)$.
As $G$ acts transitively, the particular choice of basepoint is immaterial and often notationally suppressed. 
A \emph{morphism of geometries} $(G,X)\to(H,Y)$ is a pair $(\Phi, F)$ consisting of a group homomorphism $\Phi\colon G\to H$ with $\Phi(G_x)<H_y$
 together with a $\Phi$-equivariant smooth map $F\colon (X,x)\to (Y,y)$.
A \emph{subgeometry} of $(G,X)$ is the image of a monomorphism $(H,Y)\hookrightarrow (G,X)$; namely, a subset $Y\subset X$ together with a subgroup $H<G$ preserving and acting transitively on $Y$.
An \emph{open subgeometry} is a subgeometry with $Y\subset X$ open.
One may alternatively build the theory of Klein geometries abstractly as pairs $(G,G_x)$ of a Lie group and closed subgroup, recovering the space $X$ as $X=G/G_x$ with basepoint $G_x$.
This \emph{automorphism-stabilizer} perspective is equivalent to the \emph{group-space} definitions above, with the map $(G,(X,x))\mapsto (G,G_x)$ defining an equivalence of categories.

A geometry is \emph{effective} if the only automorphism acting trivially is the identity.
The failure to be effective is measured by the intersection of all point stabilizers $K_G=\bigcap_{x\in X}\mathsf{stab}_G(x)$, and a geometry is \emph{locally effective} if $K_G$ is discrete.
The assignment $\mathsf{E}\colon(G,X)\mapsto (G/K_G, X)$ induces an equivalence of categories onto the subcategory of effective geometries; we say two geometries are \emph{effectively equivalent} if their images under $\mathsf{E}(\cdot)$ are isomorphic.
As commonplace, we switch between effectively equivalent geometries when convenient.

\subsection{Geometric Structures \& Collapse}

A $(G,X)$ structure on a manifold $M$ is defined by a maximal atlas of $X$-valued charts on $M$ with transition maps in $G$.
The set of such structures is denoted $\mathcal{S}_{(G,X)}(M)$.
Pulling an atlas back to the universal cover and analytically continuing a chosen base chart provides an alternative definition via a \emph{developing pair}: an immersion $f\colon \widetilde{M}\to X$ called the developing map, equivariant with respect to the holonomy homomorphism $\rho\colon \pi_1(M)\to G$.
A $(G,X)$ structure on $M$ only determines such a developing pair up to the action of $G$ by $g.(f,\rho)=(g.f,g\rho g\inv)$, identifying $\mathcal{S}_{(G,X)}(M)$ with the set of $G$ orbits of developing pairs under this action.
$\mathcal{S}_{(G,X)}(M)$ inherits the quotient topology from the space of developing pairs topologized by uniform convergence on compact sets.
The \emph{deformation space} of $(G,X)$ structures $\mathcal{D}_{(G,X)}(M)$ is the result of further identifying isotopy classes of structures.
More precisely, let $\mathsf{Diff}_0(M)$ denote the diffeomorphisms of $M$ isotopic to the identity, and $\widetilde{\mathsf{Diff}}_0(M)$ their lifts to $\pi_1(M)$-equivariant maps $\widetilde{M}\to\widetilde{M}$.
Then $\mathcal{D}_{(G,X)}(M)$ is the quotient of $\mathcal{S}_{(G,X)}(M)$ by the action of $\widetilde{\mathsf{Diff}}_0(M)$ by precomposition on the developing map factor.
For further reference, more detailed accounts of deformation space can be found in \cite{Goldman10,Goldman88,Baues14}.
A subgeometry $(H,Y)<(G,X)$ induces a map $\mathcal{D}_{(H,Y)}(M)\to \mathcal{D}_{(G,X)}(M)$ by viewing a $(H,Y)$ structure up to $(G,X)$ equivalence, called \emph{weakening}.  Note this map is rarely injective; for example weakening Euclidean to affine structures collapses the entirety of $\mathcal{D}_{\E^2}(T^2)$ to a point.
Dually, a developing pair for a $(G,X)$ structure with holonomy image in $(H,Y)$ can be \emph{strengthened} to an $(H,Y)$ structure by only considering equivalence up to $H$-conjugacy.

A sequence of geometric structures \emph{degenerates} if the developing maps fail to converge to an immersion even after adjusting by diffeomorphisms of $M$ and coordinate changes in $G$.
Of particular interest are \emph{collapsing degenerations}: with developing maps converging to a submersion into a lower-dimensional submanifold and holonomies limiting to a representation into the subgroup preserving this submanifold.
A trivial example is given by the collapse of Euclidean manifolds under volume rescaling.  Given a Euclidean structure $(f,\rho)$ on a manifold $M^n$ and any $r\in\R_+$, the developing pair $(rf, r\rho)$ describes the rescaled manifold with volume $r^n$ times that of the original.  As $r\to 0$ these structures collapse to a constant map and the trivial holonomy.
More interesting examples include the collapse of hyperbolic structures onto a codimension-1 hyperbolic space as studied by Danciger \cite{Danciger11,Danciger11Ideal,Danciger13} and the collapse of hyperbolic and spherical structures in \cite{Porti10,Porti98}.

Collapsing geometric structures can often be `saved' by allowing more flexible coordinate changes.
If a geometry $(H,Y)$ can be realized as an open subgeometry of $(G,X)$ then a sequence $(f_n,\rho_n)$ of collapsing $(H,Y)$ structures may actually converge \emph{as $(G,X)$ structures}, meaning there are $g_n\in G$ such that the developing pairs $g_n.(f_n,\rho_n)$ converge to a $(G,X)$ developing pair $(f_\infty, \rho_\infty)$.
When $f_\infty$ has image in an open subset $Z\subset X$ and $\rho_\infty$ maps into the subgroup $L<G$ of $Z$-preserving transformations, this $(G,X)$ developing pair strengthens to an $(L,Z)$ structure.
It is tempting to say that \emph{within} $(G,X)$ these $(H,Y)$ structures converge to an $(L,Z)$ structure.
Formalizing this notion motivates the field of \emph{transitional geometry}.

\subsection{Geometric Transitions}

A \emph{geometric transition} is a continuous path of geometries $(H_t,Y_t)$ each isomorphic to a fixed geometry $(H,Y)$, which converge to a geometry $(L,Z)\not\cong (H,Y)$.
This is difficult to define in full generality, but for our purposes it suffices to formalize geometric transitions occurring \emph{as subgeometries of a fixed ambient geometry}.
Subgeometries $(H,H_x)$ of $(G,G_x)$ correspond directly to closed subgroups $H<G$ (with $H_x=H\cap G_x$), providing a natural topology on the space of subgeometries of $(G,X)$.
The hyperspace $\mathfrak{C}_G$ of closed subgroups of a compact Lie group $G$ admits the Hausdorff metric inducing a topology in which $\{Z_n\}$ converges to the set the set of all sub-sequential limits of sequences $\{z_n\}\in Z_n$. 
This generalizes to all Lie groups $G$ by equipping $\mathfrak{C}_G$ with the topology of Hausdorff convergence on compact sets, otherwise known as the Chabauty topology \cite{Chabauty50}.

\begin{definition}
Given a geometry $(G,(X,x))$, the space of open subgeometries $\mathfrak{S}_{(G,X)}$ is defined by  $\mathfrak{S}_{(G,X)}=\{(H,H\cap G_x) \mid H<G \;\&\; \dim H-\dim(H\cap G_x)=\dim G-\dim G_x \}$ equipped with the subspace topology from $\mathfrak{C}_G\times \mathfrak{C}_{G_x}$.
\end{definition}

\begin{definition}
A \emph{continuous path of subgeometries} of $(G,X)$ is a continuous map $I\to \mathfrak{S}_{(G,X)}$.  
A geometry $(L,Z)$ is a \emph{degeneration of $(H,Y)$ in $(G,X)$} if there is a continuous path $\gamma\colon[0,1]\to\mathfrak{S}_{(G,X)}$ with $\gamma(t)\cong (H,Y)$ for $t\neq 0$ and $\gamma(0)\cong (L,Z)$.
A geometry $(L,Z)$ is a \emph{transitional geometry from $(H,Y)$ to $(H',Y')$ in $(G,X)$} if it is a degeneration of both $(H,Y)$ and $(H',Y')$.	
\end{definition}

The automorphisms $G$ of the ambient geometry act on the space of subgeometries by $g.(H,Y)=(gHg\inv, g.Y)$.
A degeneration which occurs as the limit of a sequence $g_t.(H,Y)$ for $g_t\in G$ is called a \emph{conjugacy limit} of $(H,Y)$ in $(G,X)$.
This provides the necessary background to formally consider the degeneration and regeneration of geometric structures.

\begin{definition}
Fix an ambient geometry $(G,X)$ and a subgeometry $(H,Y)$.  Then a collapsing sequence of $(H,Y)$ structures $(f_t,\rho_t)$ on a manifold $M$ \emph{degenerates to an $(L,Z)$ structure} if there is a path $g_t\in G$ with $g_t.(H,Y)\to (L,Z)$ such that $g_t.(f_t,\rho_t)$ converges as developing pairs.
Dually, an $(L,Z)$ structure on $M$ is said to \emph{regenerate into $(H,Y)$} if there such a collapsing path of $(H,Y)$ structures exists.
\end{definition}

Danciger develops \emph{half-pipe geometry} \cite{Danciger11} as half-pipe structures are the limits of the aforementioned collapse of hyperbolic conemanifolds onto codimension-1 hyperbolic space, and together with Gu\'eritaud and Kassel studies regenerations of $\mathsf{AdS}$ spacetimes from flat spacetimes \cite{DanGK16}.
Hodgson \cite{Hodgson86} and Porti 
\cite{Porti98} analyze Euclidean limits resulting from hyperbolic conemanifolds collapsing to a point, which plays an important role in the Orbifold Theorem of Cooper, Hodgson, \& Kerckhoff \cite{CooperHK00} and Boileau, Leeb \& Porti \cite{PortiLB05}.
Further work of Porti studies the nonuniform collapse of hyperbolic structures and regenerations of Nil \cite{Porti03} and Sol \cite{PortiHS01}, and the work of Ballas, Cooper \& Leitner concerns the degeneration of cusps in projective space \cite{Leitner16-Cusp,BallasCL17}.

\subsection{An Example: The Spherical-to-Hyperbolic Transition}

As a final installment of introductory material, we introduce models of the constant curvature geometies $\S^2,\E^2$ and $\H^2$ as subgeometries of projective space, and then construct a geometric transition from spherical to hyperbolic space via conjugacy limit.

\begin{definition}
As subgeometries of projective space, the constant curvature geometries are realized by the following three models.
\begin{itemize}
\item $\mathbb{S}^2=(\SO(3),\RP^2)$.  
This twofold quotient of the unit sphere is often called the \emph{elliptic plane} in older literature.
\item 	$\mathbb{E}^2=(\Euc(2),\A^2)$ with $\Euc(2)=\smat{\SO(2)&\R^2\\0&1}$ the Euclidean group acting transitively on the affine patch $\A^2=\{[x:y:1]\}\subset\RP^2$.
\item $\mathbb{H}^2=(\SO(2,1),\D^2)$ with $\D^2=\{[x:y:1]\mid x^2+y^2<1\}$ the unit disk in the affine patch $\A^2$.
\end{itemize}

Note the projective point $p=[0:0:1]$ lies in each of the above models, and the stabilizing subgroup of $p$ is equal in all three geometries to $S=\smat{\SO(2)&0\\0&1}$.
\end{definition}

Often the underlying spaces of these geometries will often be denoted $\S^2,\E^2$ and $\H^2$ as well to remind us of the inherent geometric structure.
On the level of curvature one can easily imagine producing a transition from (a small patch of) spherical space to (a small patch of) hyperbolic space through Euclidean geometry by appropriately varying the Riemannian metric.
Below we give an example realizing this transition as a \emph{conjugacy limit} connecting the three specific models above within an ambient copy of $\RP^2$.

From the group stabilizer perspective, the models above are given by the points $\S^2=(\SO(3),S),\E^2=(\Euc(2),S)$ and $\H^2=(\SO(2,1),S)$ in the space $\mathfrak{S}_{\RP^2}$ of subgeometries of the projective plane.
Let $C_t=\diag(1,1,t)$ and define the following path $\gamma\colon[-1,1]\to\mathfrak{C}_{\GL(3;\R)}$:

$$
\gamma(t)=
\begin{cases}
C_t.(\SO(2,1),S) & t<0	\\
(\Euc(2),S) & t=0\\
C_t.(\SO(3),S) & t>0
\end{cases}
$$

The point stabilizer subgroup $S$ is invariant under $C_t$ conjugacy; thus checking the continuity of $\gamma$ reduces to considering the limits of $C_t.\SO(3)$ and $C_t.\SO(2,1)$ in $\mathfrak{C}_{\GL(3;\R)}$.
The fact that each of these paths has limit $\Euc(2)$ as $t\to 0$ is a straightforward computation in the Lie algebra, a reduction which is justified by \cite{CooperDW14} Proposition 3.1 as both are conjugacy limits of algebraic groups.
Thus $\gamma$ realizes a continuous transition as subgeometries of $\RP^2$ from $\gamma(-1)=\H^2$ to $\gamma(1)=\S^2$ through $\gamma(0)=\E^2$.

\section{Heisenberg Geometry}

The Heisenberg plane is not a metric geometry but supports other familiar geometric quantities.
The standard area form $dA=dx\wedge dy$ on $\R^2$ is invariant under the action of $\Heis_+$, furnishing $\Hs^2$ with a well-defined notion of area.
The one form $dy$ is $\Heis_0$ invariant, and induces a $\Heis$-invariant foliation of $\Hs^2$ by horizontal lines together with a transverse measure.  
As a subgeometry of the affine plane, $\Hs^2$ inherits an affine connection and notion of geodesic. A curve $\gamma$ is a geodesic if $\gamma''=0$, tracing out a constant speed straight line in $\Hs^2$.

Heisenberg geometry arises as a limit of the constant curvature spaces $\S^2,\H^2$ and $\E^2$ by `zooming into while unequally stretching' a projective model.
Details can be reconstructed from \cite{CooperDW14}, and the precise characterization is reviewed in Section 4.
Here we briefly explore one degeneration of hyperbolic space to the Heisenberg plane as subgeometries of $\RP^2$.
Acting on $\H^2\in\mathfrak{S}_{\RP^2}$ by the path $A_t=\diag(t^2,t,1)$ results in a path of subgeometries $A_t\H^2$ isomorphic to the hyperbolic plane with underlying space the origin-centered ellipsoid in $\A^2$ with semimajor,semiminor axes of lengths $t^2,t$ parallel to the $x,y$ axes respectively.
As $t$ tends to infinity, the limit of these domains is $\A^2$ and the groups $A_t\O(2,1)A_t\inv$ limit to $\Heis$.
The aforementioned invariant foliation on $\Hs^2$ is a remnant of this stretching, and is parallel to the limiting direction of the major axes of $A_t\H^2$.

Unlike the degeneration of $\S^2$ and $\H^2$ to Euclidean space, the uneven stretching required to produce a Heisenberg limit distorts even the point stabilizer subgroups, which become noncompact in the limit.
Conjugation by $A_t$ stretches the circle $S=\smat{\SO(2)&0\\0&1}\subset\mathsf{M}(3;\R)$ into ellipses of increasing eccentricity limiting to the parallel lines $\smat{1&\pm x \\0&1}$ in the upper $2\times 2$ block.
As a consequence, the role of the unit tangent bundle in the constant curvature geometries is replaced for the Heisenberg plane by an appropriate space of based lines.
Indeed let $\mathcal{L}=\mathbb{P}\mathsf{T}(\Hs^2)$ be the space of pointed lines in the Heisenberg plane, and $\mathcal{H}\subset\mathcal{L}$ those belonging to the invariant horizontal foliation.  
The action of $\Heis_0$ on the plane extends to a simple transitive action on $\mathcal{L}\smallsetminus\mathcal{H}$, analogous to the action of $\Isom(\mathbb{X})$ on the unit tangent bundle $\mathsf{UT}(\mathbb{X})$ for $\mathbb{X}\in\{\H^2,\E^2,\S^2\}$.
The noncompactness of point stabilizers is sufficient to preclude  an invariant Riemannian metric, but moreover the existence of shears in the automorphism group of $\Heis$ forces any continuous $\Heis$-invariant map $d:\R^2\times\R^2\to\R$ to be constant along the lines $\{x\}\times\R$ in both factors of the domain, so there are no continuous $\Heis$-invariant distance functions at all.

\subsection{Heisenberg Structures on Orbifolds}

As a subgeometry of the affine plane, every Heisenberg structure on an orbifold $\mathcal{O}$ canonically weakens to an affine structure.
This provides strong restrictions on which orbifolds can possibly admit Heisenberg structures. It follows from a result of Benzecri that closed affine orbifolds have Euler characteristic zero \cite{Benz}; an additional self contained proof appears in \cite{Baues14}.
The deformation space of affine tori has been computed \cite{Baues14}, and weakening Heisenberg structures to affine structures provides a (non-injective) map $\omega\colon\mathcal{D}_{\Hs^2}(T^2)\to\mathcal{D}_{\A^2}(T^2)$.
Each Heisenberg orbifold inherits an area form from $\Hs^2$ and has a well defined finite total area.  The group $\R_+$ of homotheties of the plane acts on $\mathcal{D}_{\Hs^2}(\mathcal{O})$ sending an orbifold $\mathcal{O}$ with total area $\alpha$ to an orbifold $r.\mathcal{O}$ with area $r^2\alpha$, allowing
 the deformation space to be easily recovered from the space of unit area structures.
 
 \begin{observation}
 \label{obs:Area}
 The action of $\R_+$ by homotheties on the plane induces an action on $\mathcal{D}_{\Hs^2}(\mathcal{O})$ defined by $r.[f,\rho]=[rf,r\rho]$.  
 This gives a homeomorphism $\mathcal{D}_{\Hs^2}(\mathcal{O})=\R_+\times\mathcal{T}_{\Hs^2}(\mathcal{O})$ for $\mathcal{T}_{\Hs^2}(\mathcal{O})$ the subspace of unit area structures, analogous to the Techim\"uller space for Euclidean tori.
 \end{observation}

As $dy$ is invariant under the action of $\Heis_0$, any Heisenberg surface with holonomy into $\Heis_0$ inherits a closed nondegenerate 1-form and corresponding foliation.  
This observation leads to a self-contained proof that every Heisenberg orbifold has vanishing Euler characteristic, simple enough that we include it for completeness.

\begin{proposition}
Every closed Heisenberg orbifold is finitely covered by a torus with holonomy in $\Heis_0$.
\end{proposition}
\begin{proof}
Let $\mathcal{O}$ be a Heisenberg orbifold, with developing map $f\colon\widetilde{\mathcal{O}}\to\Hs^2$ and holonomy $\rho\colon\pi_1(\mathcal{O})\to\Heis$.
As $f$ immerses $\widetilde{\mathcal{O}}$ in the plane it has no singular locus; thus $\widetilde{\mathcal{O}}$ is a manifold and $\mathcal{O}$ is good.
By the classification of two dimensional orbifolds then, $\mathcal{O}$ is not the spindle or teardrop, and is finitely covered by some surface $\Sigma\to\mathcal{O}$.
The Heisenberg structure on $\mathcal{O}$ pulls back to $\Sigma$ with developing pair $(f,\rho|_{\pi_1(\Sigma)})$.
Passing to an at most 4-sheeted cover, we may assume the holonomy of $\Sigma$ takes values in $\Heis_0$.
Thus $\Sigma$ inherits a nondegenerate 1-form $\omega\in\Omega^1(\Sigma)$ from $dy$ on $\Hs^2$.  Choose a Riemannian metric $g$ on $\Sigma$.  
Then $\omega$ defines a non-vanishing vector field $X_\omega$ by $\omega(\cdot)=g(X_\omega,\cdot)$, and so $\chi(\Sigma)=0$.  As $\Heis_0$ acts by orientation preserving transformations, $\Sigma$ is a torus.
\end{proof}

Thus Heisenberg tori with holonomy in $\Heis_0$ play a fundamental role to the classification of Heisenberg orbifolds, and it is natural to study them first.
By the previous observation, in particular it suffices to study the Teichm\"uller space of unit area structures, whose holonomy are determined up to conjugacy and homotheties of the plane.

\subsection{Representations of $\Z^2$ into $\Heis$}

To classify tori with holonomy into $\Heis_0$ we compute the representation variety $\mathcal{R}=\Hom(\Z^2,\Heis_0)$.
The quotients of $\mathcal{R}$ by homothety and Heisenberg conjugacy are denoted $\mathcal{H}=\mathcal{R}/\R_+$ and $\mathcal{X}=\mathcal{R}/\Heis_0$ respectively.
The holonomies of unit area structures lie in the double quotient $\mathcal{U}=\mathcal{X}/\R_+\cong\mathcal{H}/\Heis_0$.
Representations into the center of $\Heis_0$ act by collinear translations on $\Hs^2$, and a simple argument of section 3.3 precludes these from being the holonomy of any Heisenberg structure.
Thus, we are primarily concerned with the subset $\mathcal{R}^\star\subset\mathcal{R}$ of representations not into the center, and its quotients $\mathcal{X}^\star\subset\mathcal{X},\mathcal{H}^\star\subset\mathcal{H}$ and $\mathcal{U}^\star\subset\mathcal{U}$.
Explicitly dealing with these representation spaces is easiest using coordinates from the Lie algebra, introduced below.

\begin{proposition}
The map $\log\colon \Heis_0\to\heis$ induces an isomorphism of varieties \\
$\Hom(\Z^2,\Heis_0)\cong\Hom(\R^2,\heis)$.
\end{proposition}
\begin{proof}
Both $\Heis_0$ and $\heis$ inherit their structure as algebraic varieties from their inclusion in the affine space $\mathsf{M}(3,\R)$ of $3\times 3$ real matrices.
As $\heis$ is nilpotent, the power series $\exp\colon\heis\to\Heis_0$ terminates, and thus is algebraic.
Indeed, $\exp$ is an isomorphism of varieties with polynomial inverse $\log\colon\Heis_0\to\heis$.
Recall that evaluation on the generators $e_1,e_2\in\Z^2\subset\R^2$ identifies the collections of representations with subvarieties of $\Heis_0\times\Heis_0$, $\heis\times\heis$ respectively.
Applying the exponential/logarithm coordinatewise provides the required algebraic isomorphism $\Hom(\Z^2,\Heis_0)\cong\Hom(\R^2,\heis)$.

\begin{center}
\includegraphics[width=0.5\textwidth]{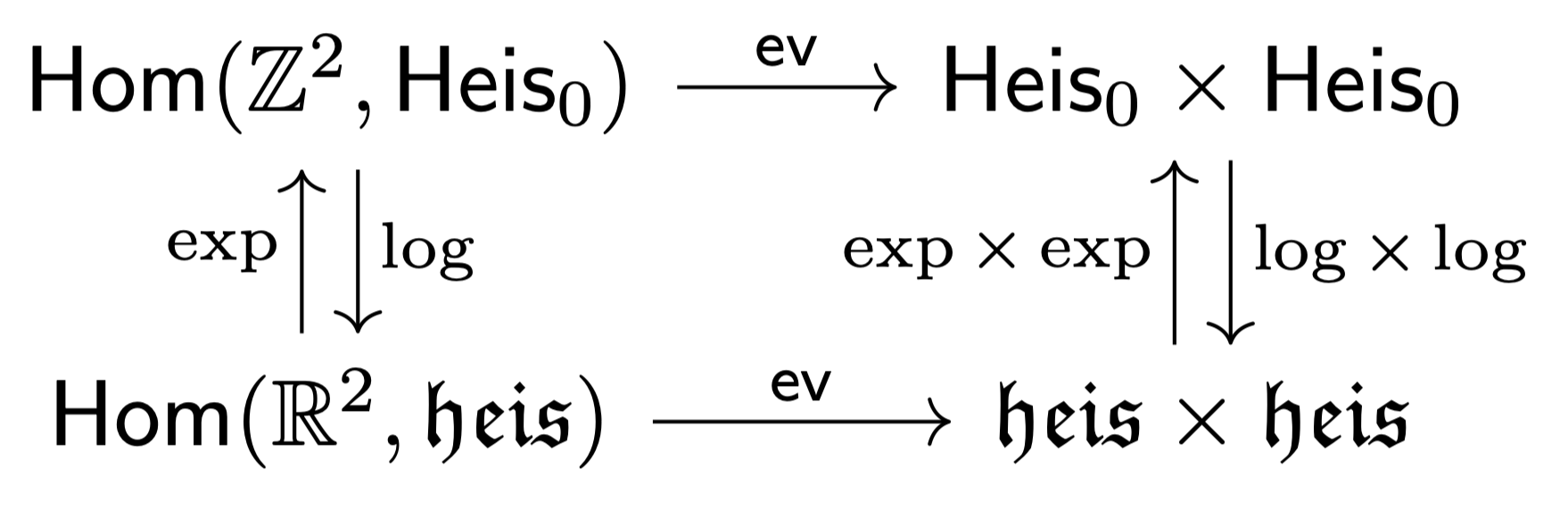}
\end{center}

\end{proof}

We continue to denote the induced isomorphisms $\mathcal{R}\cong\Hom(\R^2,\heis)$ by $\exp$ and $\log$,
and call the vector $(\vec{x},\vec{y},\vec{z})\in\R^6$ the \emph{Lie algebra coordinates} for the representation $\rho\in\mathcal{R}$ 
when $\mathsf{ev}(\log\rho)=\left(\smat{x_1 & z_1\\&y_1},\smat{x_2 & z_2\\&y_2}\right)$.

\begin{proposition}
$\mathcal{R}$ is isomorphic to $V(x_1y_2-x_2y_1)\times\R^2$. 
\end{proposition}
\begin{proof}
Evaluation on the generators identifies the representation variety $\Hom(\R^2,\heis)$ with the kernel of the Lie bracket $[\cdot,\cdot]\colon\heis^2\to\heis$.	
Indeed $\left[\smat{x_1&z_1\\&y_1},\smat{x_2&z_2\\&y_2}\right]=\smat{0 & x_1y_2-x_2y_1\\&0}$,
so $\ker[\cdot,\cdot]$ is cut out precisely by $x_1y_2=x_2y_1$ in $\heis^2$ and $(\vec{x},\vec{y},\vec{z})\in\R^6$ is the Lie algebra coordinates of a representation $\rho\in\mathcal{R}$ if and only if $(\vec{x},\vec{y})\in V(x_1y_2-x_2y_1)$ and $(z_1,z_2)\in\R^2$.
\end{proof}

\begin{proposition}
The space $\mathcal{H}^\star=\mathcal{R}^\star/\R_+$ of representations modulo homothety with image not contained in the center of $\Heis$ is homeomorphic to $\R^2\times T^2$.	
\end{proposition}
\begin{proof}
Denote by $\R^2_{(\vec{x},\vec{y})}$ the fiber above $(\vec{x},\vec{y})$ under the projection $(\vec{x},\vec{y},\vec{z})\mapsto(\vec{x},\vec{y})$.
The hypersurface $V=V(x_1y_2-x_2y_1)$ has one singularity at $0$, above which $\R^2_{(0,0)}$ consists of the representations into the center.
Homotheties of $\Hs^2$ induce the $\R_+$ action $t.(\vec{x},\vec{y},\vec{z})=(t\vec{x},t\vec{y},t\vec{z})$ on $\mathcal{R}$; thus $V\subset\R^4$ is a cone and $\mathcal{H}^\star$ identifies with the product of $\R^2$ with the intersection $V\cap\S^3$.
The change of coordinates on $\R^4$ given by $(x_1,x_2,y_1,y_2)=(u_1+v_1,v_2+u_2,v_2-u_2,u_1-v_1)$ provides an isomorphism $V\cong V(u_1^2+u_2^2-v_1^2-v_2^2)$ 
identifying $V\cap \S^3$ with the Clifford torus $T=\{(u,v)\in\mathbb{C}^2\colon\,\|\vec{u}\|=\|\vec{v}\|=1/\sqrt{2}\}$, verifying the claim.
\end{proof}

\begin{corollary}
The section of $\mathcal{R}^\star\to\mathcal{H}^\star$ sending each homothety class $[\rho]_{\R_+}=[(\vec{x},\vec{y},\vec{z})]_{\R_+}$ to the representative with $(\vec{x},\vec{y})\in T^2\subset \S^3$ is a diffeomorphism of $\mathcal{H}^\star$ onto its image.  We use this to identify $\mathcal{H}^\star$ with the algebraic variety
$V(x_2y_1-x_1y_2,\|x\|^2+\|y\|^2-1)\subset\R^6$.
\end{corollary}

We have identified the space $\mathcal{R}$ of all representations as a product $\R^2\times V$ of a plane with a cone on the torus, with representations into the center parameterized by the plane above the cone point of $V$.
Restricting to representations \emph{not} into the center, it proves useful to remove this cone point, and consider the space $V\smallsetminus\{0\}\cong \R_+\times T^2$, which we denote $V^\star$ to remain consistent with other notations.

\begin{proposition}

Let $\mathcal{X}^\star$ be the conjugacy quotient $\mathcal{X}^\star=\mathcal{R}^\star/\Heis_0$.
Then the function $\pi\colon \mathcal{X}^\star\to V^\star$ defined by sending the $\Heis_0$ orbit of $\rho=(\vec{x},\vec{y},\vec{z})\in\mathcal{R}^\star$ to $(\vec{x},\vec{y})\in V^\star$ equipps $\mathcal{X}^\star$ with the structure of a line bundle over $V^\star$.
Topologically we can identify this line bundle up to isomorphism by noting that it is once-twisted above each generator of $\pi_1(V^\star)=\Z^2$.
\end{proposition}
\begin{proof}

A computation reveals the conjugation action of $\Heis_0$ on $\mathcal{R}$ in Lie algebra coordinates is expressed
$\smat{1& g&k\\&1&h\\&&1}.(\vec{x},\vec{y},\vec{z})=(\vec{x},\vec{y},\vec{z}+g\vec{y}-h\vec{x})$.
Thus $\Heis_0$ acts trivially on the first factor of $\mathcal{R}=V\times\R^2$ and the orbit of a point $\vec{z}\in\R^2_{(\vec{x},\vec{y})}$ 
is the coset of $\mathsf{span}\{\vec{x},\vec{y}\}\subset\R^2_{(\vec{x},\vec{y})}$ containing it.
In the subset $\mathcal{R}^\star$ at least one of $\vec{x},\vec{y}$ is nonzero, and the condition that $(\vec{x},\vec{y})\in V(x_1y_2-x_2y_1)=V(\det\smat{x_1&y_1\\x_2&y_2})$ implies $\vec{x}$ and $\vec{y}$ are linearly dependent.
It follows that the $\Heis_0$ orbits on $\mathcal{R}^\star$ are lines, foliating each $\R^2_{(\vec{x},\vec{y})}$ over $V^\star$ and the leaf space is a line bundle over $V^\star$.

Equipping each $\R^2_{(\vec{x},\vec{y})}$ with the standard inner Euclidean inner product, 
a canonical choice of representatives for cosets of $\ell_{(\vec{x},\vec{y})}=\mathsf{span}\{\vec{x},\vec{y}\}$ is given by the orthogonal line $\ell_{(\vec{x},\vec{y})}^\perp\subset\R^2_{(\vec{x},\vec{y})}$.
This defines a section $\mathcal{X}^\star\to\mathcal{R}^\star$ sending a conjugacy class $[\rho]_{\Heis_0}=[(\vec{x},\vec{y},\vec{z})]_{\Heis_0}$ to its representation with $\vec{z}$-coordinate on $\ell_{(\vec{x},\vec{y})}^\perp$, and identifies
 $\mathcal{X}^\star=\{(\vec{x},\vec{y},\vec{z})\mid (\vec{x},\vec{y})\in V^\star, \; \vec{z}\in\ell_{(\vec{x},\vec{y})}^\perp\}$ with a subbundle of $V^\star\times\R^2\to V^\star$.

Line bundles over $V^\star\cong\R_+\times T^2$ are in bijection with $H^1(T^2,\Z_2)\cong\Z_2^2$, determined up to isomorphism by whether pulling back along generators of $\pi_1(T)^2$ gives cylinders or M\"obius bands.
A convenient choice of generators in the $(\vec{u},\vec{v})$ coordinates introduced above are $\alpha(\theta)=(\vec{e_1},\vec{p_\theta})$ and $\beta(\theta)=(\vec{p_{\theta}},\vec{e_1})$ for $e_1=\smat{1\\0}$ and $\vec{p_\theta}=\smat{\cos\theta\\sin\theta}$.
An explicit computation using the description of $\mathcal{X}^\star$ above shows the bundle restricts to a M\"obius band above each of $\alpha,\beta$, so $\mathcal{X}^\star$ is the  line bundle over $\R_+\times T^2$ represented by $(1,1)\in H^1(T^2,\Z_2)$.
\end{proof}

The choice of explicit sections has identified $\mathcal{H}^\star$ and $\mathcal{X}^\star$ with subsets of $\mathcal{R}$.
The space of interest $\mathcal{U}^\star$ identifies with their intersection, $\mathcal{X}^\star\cap\mathcal{H}^\star$, which is the restriction of $\mathcal{X}^\star\to V^\star$ to the base $T^2\subset\S^3$.

\begin{corollary}
\label{cor:U}
Let $\mathcal{U}^\star$ denote the quotient of $\mathcal{R}^\star$ by homothety and conjugacy (equivalently, the quotient of $\mathcal{X}^\star$ by homothety).
Then the map $\mathcal{U}^\star\to T^2$ defined by sending the orbit of $\rho=(\vec{x},\vec{y},\vec{z})$ to $(\vec{x}/\|\vec{x}\|,\vec{y}/\|\vec{y}\|)\in V\cap \S^3\cong T^2$ equips $\mathcal{U}^\star$ with the structure of a line bundle over the torus.
We may realize $\mathcal{U}^\star$ explicitly the subvariety of $\mathcal{U}^\star\subset\R^6$ consisting of triples of vectors $(\vec{x},\vec{y},\vec{z})$ such that $\vec{x}$ and $\vec{y}$ are collinear, and $\vec{z}$ is orthogonal to their span.
$$\mathcal{U}^\star=
V\left(\mat{
\|x\|^2+\|y\|^2=1,& \vec{z}\cdot\vec{x}=0\\
x_1y_2-x_2y_1=0,& \vec{z}\cdot\vec{y}=0
}\right)\subset\R^6$$
Note that like $\mathcal{X}^\star$, we may characterize the bundle $\mathcal{U}^\star\to T^2$ topologically by noting that its restriction to each standard generator of $T^2$ is a M\"obius band. 
\end{corollary}

The developing pair of a Heisenberg torus is only well defined up to orientation preserving transformations, so potential holonomies lie in the space $\mathcal{R}/\Heis_+$, a twofold quotient of $\mathcal{U}^\star$ computed here.
We will deal with this $\Z_2=\Heis_+/\Heis_0$ ambiguity after determining which points of $\mathcal{U}^\star$ are in fact holonomies.

\subsection{The Deformation Space of Tori}

As a warm-up to computing the deformation space of Heisenberg tori, we review the analogous problem for Euclidean and affine structures.
Euclidean tori are complete metric spaces, and so are determined by their holonomy, which is necessarily discrete and faithful (for instance, by Thurston's book \cite{Thurston80}, Proposition 3.4.10). 
Discrete subgroups $\Z^2<\Isom(\E^2)$ act by translations,
thus the deformation space of Euclidean tori identifies with the $\Isom(\E^2)$-conjugacy classes of marked planar lattices, $\mathcal{D}_{\E^2}(T^2)\cong \GL(2;\R)/\O(2)$.
The unit area structures parameterized by the familiar Teichm\"uller space $\H^2=\SL(2;\R)/\SO(2)$.

The affine plane admits no invariant metric, which complicates the story significantly.
Complete affine structures have universal cover affinely diffeomorphic to $\A^2$, but in contrast to the Euclidean case incomplete structures abound.
The work of Baues \cite{Baues14} provides a remarkably comprehensive description of the classification of affine tori, in particular containing the following classification theorem.

\begin{theorem}[\cite{Baues14}, Theorem 5.1]
The universal cover of an affine torus is affinely diffeomorphic to one of the following spaces: the affine plane $\A^2$, the half plane $\mathcal{H}=\{(x,y)\mid y>0\}$, the quarter plane $\mathcal{Q}=\{(x,y)\in\A^2\mid x,y>0\}$ or the universal cover of the punctured plane $\mathcal{P}=\widetilde{\A^2\smallsetminus 0}$.
Furthermore the developing maps of affine structures are covering projections onto their images.	
\end{theorem}

As $\Hs^2$ admits no invariant metric, we must be prepared for complications similar to the affine case.
Such difficulties do not materialize however, as canonically weakening Heisenberg structures to affine ones, we may use the classification above to show all Heisenberg tori are complete.

\begin{corollary}
All Heisenberg structures on the torus are complete.
\label{cor:Complete}
\end{corollary}
\begin{proof}
Let $(f,\rho)$ be the developing pair for a Heisenberg torus $T$, considered as an affine structure.
If $T$ is not complete, there is an affine transformation $A$ with $A.f(\widetilde{T})\in\{\mathcal{H},\mathcal{Q},\A^2\smallsetminus 0\}$ and holonomy $A\rho A\inv$ preserving this developing image.
But by the classification of affine tori, holonomies of these tori contain elements of $\det\neq 1$, whereas $\Heis$ is unipotent so $\det A\rho(\Z^2)A\inv=\{1\}$.
Thus $T$ is in fact complete, with developing map a diffeomorphism $f\colon \widetilde{T}\to \A^2$.
\end{proof}

Here we pursue a self-contained computation the deformation space $\mathcal{D}_{\Hs^2}(T^2)$, using the understanding of representations $\Z^2\to \Heis_0$ up to conjugacy developed in section 3.1.
Specifically, for $\rho\in\Hom(\Z^2,\Heis)$ we either construct a corresponding developing map $f$ giving a Heisenberg structure $(f,\rho)$ on $T^2$ (and prove its uniqueness), or we show no developing map for $\rho$ can exist.

A developing map for $\rho\colon\Z^2\to\Heis$ is a $\rho$-equivariant immersion $f\colon\R^2\to\Hs^2$.
A natural $\rho$-equivariant self map of the plane can be constructed directly from $\rho$, relying on the fact that
each representation of $\Z^2$ extends uniquely to a representation $\hat{\rho}\colon\R^2\to\Heis_0$ via $\widehat{\rho}(x,y)=\rho(e_1)^x\rho(e_2)^y$.
The orbit map $f_\rho\colon \R^2\to \Hs^2$ defined by $(x,y)\mapsto\widehat{\rho}(x,y).\vec{0}$ for this extended representation is $\rho$-equivariant, and thus a developing map for a Heisenberg structure when it is an immersion.
As the following two propositions show, this construction actually produces developing maps for all complete Heisenberg tori (and thus by Corollary \ref{cor:Complete} for all Heisenberg tori, although with the aim of producing a self-contained proof we do not presume that here).

\begin{proposition}
Let $\mathcal{F}\subset\mathcal{U}$ be the subset of representations $\rho$ with extensions $\widehat{\rho}$ acting freely on $\Hs^2$.
Then each $\rho\in\mathcal{F}$ determines a unique Heisenberg structure on $T^2$, which is complete, and all complete structures with holonomy in $\Heis_0$ arise this way.
\end{proposition}
\begin{proof}
If $\widehat{\rho}$ acts freely, the orbit map $f_\rho\colon\R^2\to\Hs^2$ is injective, and a computation reveals $(df_\rho)_0\colon T_0\R^2\to T_0\Hs^2$ is injective.
Furthermore $(df_\rho)_x=\widehat{\rho}(x).(df_\rho)_0$ so $f_\rho$ is an immersion of $\R^2$ and $(f_\rho,\rho)$ is a developing pair for a Heisenberg torus.
Similarly, the other orbit maps $\vec{u}\mapsto\widehat{\rho}(\vec{u}).q$ are immersions (thus open maps) for any $q\in\Hs^2$, and distinct $\widehat{\rho}(\R^2)$ orbits partition $\Hs^2$ into a disjoint union of open sets.
By connectedness then $f_\rho$ is onto, hence a diffeomorphism so the corresponding Heisenberg structure is complete.

Alternatively, let $\rho\colon\Z^2\to\Heis_0$ be the holonomy of a complete torus, but assume $\widehat{\rho}\colon\R^2\to\Heis_0$ fails to act freely.
Then some element, and hence some 1-parameter subgroup $L<\R^2$, fixes a point under the action induced by $\widehat{\rho}$.  
This line $L$ intersects $\Z^2$ only in $\vec{0}$ (as $\rho$ acts freely by completeness); and so is dense in the quotient $\R^2/\Z^2$.  
Thus there are sequences $\vec{v}_n\in\Z^2$ with $\rho(v_n)$ coming arbitrarily close to stabilizing a point, and $\widehat{\rho}$ does not act properly discontinuously, contradicting completeness.

Finally, let $(f,\rho)$ be a complete structure and $(\phi,\rho)$ another structure with the same holonomy.  
Then $f\inv \phi:\widetilde{T}\to\widetilde{T}$ is $\pi_1(T)$-equivariant and descends to a diffeomorphism $\psi:T\to T$.  
But $\psi_\ast$ is the identity on fundamental groups and as the torus is a $K(\pi,1)$, $\psi$ is isotopic to the identity.
Thus $(f,\rho)$ and $(\phi,\rho)$ are developing pairs for the same Heisenberg structure.
\end{proof}

Constructing developing maps from the extensions $\widehat{\rho}$ provides endows these tori with the structure of a commutative group via the identification $\widehat{\rho}(\R^2)/\rho(\Z^2)\cong f_\rho(\R^2)/\rho(\Z^2)$.
The existence of this group structure can more generally be deduced from the similar observation of Baues and Goldman concerning affine structures \cite{BauesGoldman05}.

\begin{corollary}
Complete Heisenberg tori are the group objects in the category of Heisenberg manifolds, analogous to elliptic curves in the category of Riemann surfaces.	
\end{corollary}

\begin{proposition}
The subset $\mathcal{F}\subset\mathcal{U}$ of conjugacy classes with freely acting extensions $\widehat{\rho}\colon\R^2\to\Heis_0$ is a trivial $\R^\times$ bundle over the cylinder $\mathsf{Cyl}=T^2\smallsetminus S$, for $S$ the circle defined by the intersection of  $T^2=V(x_1y_2-x_2y_1)\cap\S^3$ with the plane $V(y_1,y_2)$.
\end{proposition}
\begin{proof}
A representation $\widehat{\rho}\in\mathcal{U}$ is faithful if and only if the logarithm of its generators $\smat{x_1&z_1\\&y_1}$ and $\smat{x_2 &z_2\\&y_2}$ are linearly independent in $\heis$.
In Lie algebra coordinates, linearly dependent elements of $\heis^2$ form the variety $\mathsf{Rk}_1\subset\mathsf{M}_{3\times 2}(\R)$ of rank one matrices $(\vec{x},\vec{y},\vec{z})=\smat{x_1&y_1&z_1\\x_2&y_2&z_2}$, alternatively described as triples of simultaneously collinear vectors $\vec{x}\parallel\vec{y}\parallel\vec{z}\in\R^2$.
There are no faithful $\R^2$ representations into the 1-dimensional center of $\Heis$, so it suffices to consider the representations in $\mathcal{U}^\star$.
Recalling \ref{cor:U}, points $(\vec{x},\vec{y},\vec{z})$ of $\mathcal{U}^\star$ satisfy $\vec{x}\parallel\vec{y}$ and $\vec{z}$ perpendicular to their span.
Thus any $(\vec{x},\vec{y},\vec{z})\in\mathcal{U}^\star\cap\mathsf{Rk}_1$ necessarily has $\vec{z}=0$, so the intersection $\mathcal{U}^\star\cap\mathsf{Rk}_1$ is the torus
$(\vec{x},\vec{y},0)\subset\mathcal{X}^\star$.
 The conjugacy classes of faithful representations constitute the complement of this zero section of $\mathcal{U}^\star\to T^2$.

A non-identity element of $\Heis_0$ stabilizes a point of $\Hs^2$ if and only if it acts trivially on the leaf space of the invariant foliation and has nontrivial shear.  In Lie algebra coordinates this forms the set $\mathcal{S}=\left\{\smat{x&z\\&0}\mid x\neq 0\right\}\subset\heis$. 
The extension $\widehat{\rho}$ acts freely if and only if in Lie algebra coordinates, each generator misses $\mathcal{S}$.
All faithful representations $(\vec{x},\vec{y},\vec{z})$ with $y_1,y_2\neq 0$ act freely, and all with $\vec{y}=0$ fail to.
If $\vec{y}=(0,y_2)$ then $\rho\in\mathcal{R}$ implies $x_1=0$ so $\rho$ acts freely, and similarly for $\vec{y}=(y_1,0)$.
Thus faithful representations fail to act freely if and only if $\vec{y}=0$, and the space of freely acting representations is $\mathcal{F}=\mathcal{U}^\star\smallsetminus V(z_1,z_2)\cup V(y_1,y_2)$.

The intersection $S=T^2\cap V(y_1,y_2)$ is a $(1,1)$ curve with respect to the $(\vec{u},\vec{v})$ coordinates, and $\mathcal{U}^\star\smallsetminus V(y_1,y_2)$ is an $\R$-bundle over $\mathsf{Cyl}=T^2\smallsetminus S$.
This bundle is trivial as the generator of $\pi_1(\mathsf{Cyl})$ is parallel to $V(y_1,y_2)$ and the restriction the doubly twisted bundle $\mathcal{X}$ to a $(1,1)$ curve in the base is a cylinder.
The subvariety $V(z_1,z_2)$ is the zero section of this bundle, thus its complement is the trivial $\R^\times$ bundle over $\mathsf{Cyl}$.
\end{proof}

This classification gives a simple, self contained argument that no incomplete structures exist.
An incomplete structure must have holonomy in $\mathcal{U}\smallsetminus\mathcal{F}$, but geometric reasons preclude these from being the holonomy of Heisenberg tori.
This completes the classification of tori with $\Heis_0$ holonomy, and a quick observation implies there can be no others.

\begin{proposition}
Representations $\rho\in\mathcal{U}\smallsetminus \mathcal{F}$ are not the holonomy of any Heisenberg torus.
Consequently all Heisenberg tori are complete, with holonomy into $\Heis_0$.
\end{proposition}
\begin{proof}
There are three classes of elements in $\mathcal{U}\smallsetminus \mathcal{F}$: representations into the center, representations $(\vec{x},\vec{y},\vec{z})$ with $\vec{z}=0$	 and representations with $\vec{y}=0$.
These classes are all topologically conjugate, and preserve a fibration of the plane $\Hs^2\surject \R$.
Representations into the center act by translations parallel to the $x$ axis, preserving the invariant foliation of $\Hs^2$, and similarly for those with $\vec{y}=0$.
Representations with $\vec{z}=0$ are not faithful, and factor through a representation $\R\to\Heis$ with orbits foliating the plane by parabolas.

To see these cannot be the holonomy of tori, let
$\rho\in\mathcal{U}\smallsetminus\mathcal{F}$ preserve the fibration $\pi\colon\Hs^2\surject\R$, and assume $(f,\rho)$ is a developing pair for some Heisenberg torus.
Let $\Omega=f(\widetilde{T})$ be the developing image, and note $\pi(\Omega)\subset \R$ is open as $f$ is a local diffeomorphism and $\pi$ is a bundle projection.
Let $Q\subset\widetilde{T}$ be a compact fundamental domain for the action of $\Z^2$ by covering transformations, and note that $\pi(f(Q))=\pi(f(\Omega))$ as $\rho$ is fiber preserving.
But $\pi(f(Q))$ is compact, and thus not open in $\R$, a contradiction.

It follows from this that all Heisenberg tori are complete, and have holonomy in $\Heis_0$.
Indeed $T$ be any Heisenberg torus with developing pair $(f,\rho)$ and $\widetilde{T}\to T$ the cover corresponding to the subgroup $\rho(\Z^2)\cap\Heis_0$.  
Then $\widetilde{T}$ is complete so $T$ is also, and $\rho(\Z^2)$ acts freely and properly discontinuously on $\Hs^2$.
As $T^2$ is orientable the holonomy takes values in $\Heis_+$, but every element of $\Heis_+\smallsetminus \Heis_0$ fixes a point in $\Hs^2$ so in fact $\rho$ is $\Heis_0$ valued and $T=\widetilde{T}$.
\end{proof}

Thus a representation $\rho\colon\Z^2\to\Heis$ is either the holonomy of a unique complete structure on $T^2$, or is not the holonomy of any geometric structure at all.
After dealing with the slight annoyance of $\Heis_0$ vs. $\Heis_+$ conjugacy, this directly provides a description of the 
the Teichm\"uller space $\mathcal{T}_{\Hs^2}(T^2)$ of unit area structures and the corresponding deformation space $\mathcal{D}_{\Hs^2}(T^2)=\R_+\times\mathcal{T}_{\Hs^2}(T^2)$.

\begin{theorem}
The projection onto holonomy identifies the Teichm\"uller space of unit area Heisenberg tori with the quotient of $\mathcal{F}$ by the free $\Z_2$ action of conjugacy by $\diag(-1,-1,1)$ and $\mathcal{T}_{\Hs^2}(T^2)\cong \mathcal{F}/\Z^2\cong \R^2\times \S^1$.
\end{theorem}
\begin{proof}
The map $\hol\colon\mathsf{Dev}_{\Hs^2}(T^2)\to\mathcal{R}$ projecting a developing pair onto its holonomy is a local homeomorphism by the Ehresmann-Thurston principle, which induces a continuous map $\overline{\hol}\colon \mathcal{D}_{\Hs^2}(T^2)\to\mathcal{R}/\Heis_+$.
The work above shows the map $\dev\colon\mathcal{F}\to\mathcal{D}_{\Hs^2}(T^2)$ defined by $\rho\mapsto [f_\rho,\rho]$ is a continuous surjection onto Teichm\"uller space $\mathcal{T}_{\Hs^2}(T^2)$.
As $\mathcal{F}\subset\mathcal{U}$ was defined only up to $\Heis_0$ conjugacy, $\dev$ factors through the quotient by $(\Heis_+/\Heis_0)\cong\Z_2$ conjugacy to a continuous bijection $\overline{\dev}\colon\mathcal{F}/\Z_2\to\mathcal{T}_{\Hs^2}(T^2)$.  
The composition $\overline{\hol}\circ\overline{\dev}$ is the identity on $\mathcal{F}/\Z_2$, so $\overline{\dev}$ is a homeomorphism.

Thus, $\mathcal{T}_{\Hs^2}(T^2)\cong\mathcal{F}/\Z_2$.
The quotient $\Heis_+/\Heis_0\cong\Z_2$, generated by $\diag(-1,-1,1)$, acts by conjugation in Lie algebra coordinates as $\diag(-1,-1,1).(\vec{x},\vec{y},\vec{z})=(\vec{x},-\vec{y},-\vec{z})$.
This action is free on $\mathcal{F}$ and the quotient $\mathcal{T}_{\Hs^2}(T^2)$ is the trivial $\R_+$ bundle over $\mathsf{Cyl}$, which is homeomorphic to the open solid torus  $\R^2\times\S^1$, and $\mathcal{D}_{\Hs^2}(T^2)\cong \R^3\times\S^1$.
\end{proof}

\begin{figure}[h!]
\includegraphics[width=\textwidth]{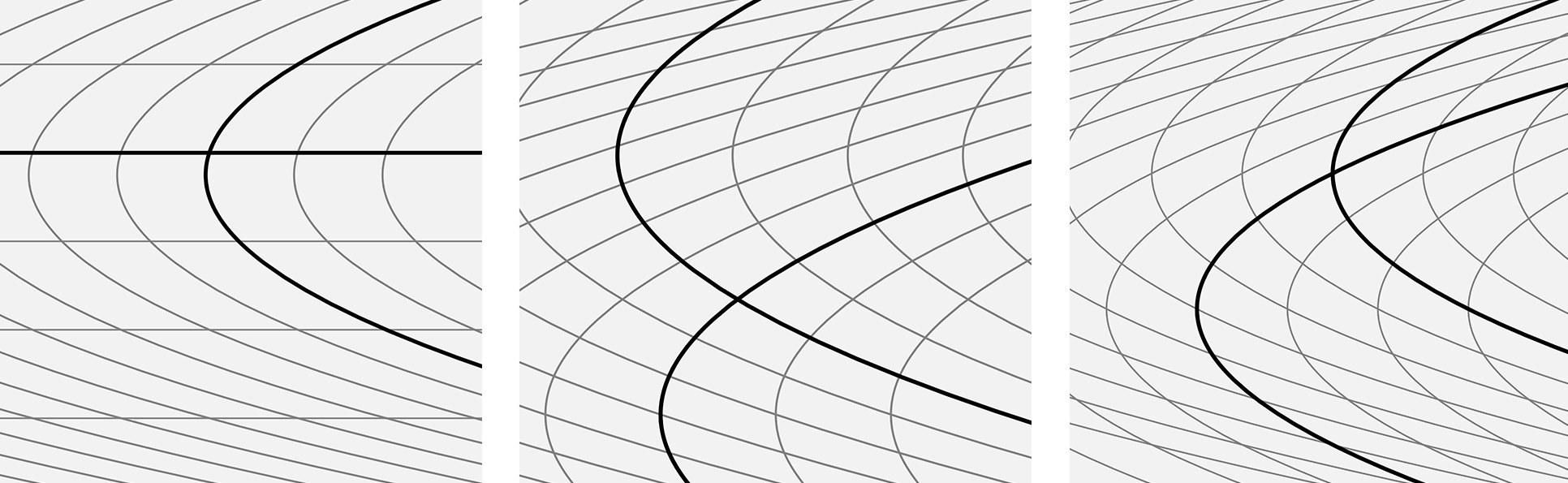}	
\caption{Some examples of developing maps for Heisenberg shear tori.}
\end{figure}

The identification $\mathcal{T}_{\Hs^2}(T^2)=\mathcal{F}/\Z_2$ identifies two distinct classes of Heisenberg tori; those containing a shear in their holonomy and those with holonomy into the subgroup of translations of the plane.
We will refer to these as \emph{shear tori} and \emph{translation tori} respectively.

\begin{corollary}
The space of unit-area translation tori is homeomorphic to $\R\times \S^1$, corresponding to the points of $\mathcal{F}\cap V(x_1,x_2)$.
\end{corollary}

It is notable that the set of developing pairs for Heisenberg translation tori is the same as the set of developing pairs for Euclidean tori, but the corresponding deformation spaces are not homeomorphic, with $\mathcal{T}_{\E^2}(T^2)$ a disk and $\mathcal{T}_{\Hs^2}(T^2)$ a cylinder.
This is due to the different notion of equivalence coming from $\Heis_+$ and $\Isom_+(\E^2)$ conjugacy; the former acting by shears and the latter by rotations.
The familiar fact that Euclidean torus has a representative holonomy containing horizontal translations is a consequence of this, as is the fact that each Heisenberg translation torus has a representative holonomy translating along (Euclidean) orthogonal lines.

\begin{figure}[h]
\includegraphics[width=\textwidth]{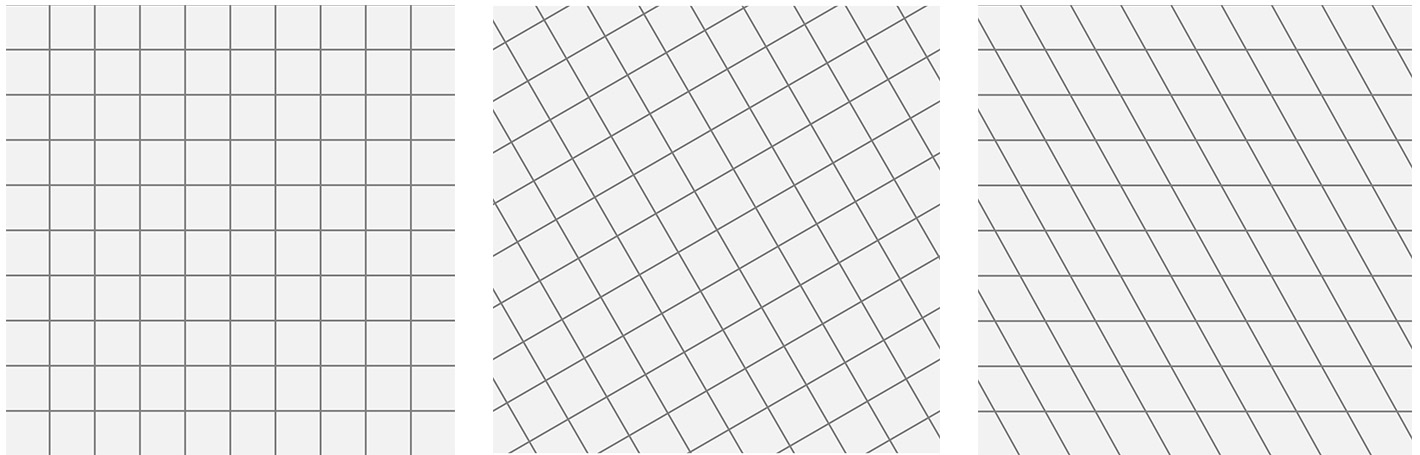}
\caption{Developing maps for translation tori.  The left two are equivalent as Euclidean structures, whereas the right two are as Heisenberg structures.
All three represent the same (unique) affine translation torus.}	
\end{figure}

Every Heisenberg structure canonically weakens to an affine structure, defining the map $\omega\colon\mathcal{D}_{\Hs^2}(T^2)\to\mathcal{D}_{\A^2}(T^2)$ with image in the complete structures.

\begin{corollary}
The space $\omega(\mathcal{D}_{\Hs^2}(T^2))$ of Heisenberg structures up to affine equivalence is one dimensional, homeomorphic to $\R$.
\end{corollary}
\begin{proof}
By Goldman and Baues \cite{BauesGoldman05}, the space of complete affine structures on $T^2$ is diffeomorphic to the plane, and by completeness we identify this with its projection onto holonomy.
This realizes $\omega(\mathcal{D}_{\Hs^2}(T^2))$ as the quotient of $\mathcal{F}$ by affine conjugacy, on which the subgroups of rotations and linearly independent scalings act freely.
Thus the $\S^1$ factor and $\R_+^2$ directions of independent scalings collapse in the quotient, and $\omega(\mathcal{D}_{\Hs^2}(T^2))\cong\R$.

 \end{proof}

\subsection{Which Orbifolds Admit Heisenberg Structures?}

We may use this description of the deformation space of tori to understand all Heisenberg orbifolds. 
An orbifold covering $\pi\colon\mathcal{Q}\to\mathcal{O}$ induces a map $\pi^\ast\colon\mathcal{D}_{\Hs^2}(\mathcal{O})\to\mathcal{D}_{\Hs^2}(\mathcal{Q})$ by pullback of geometric structures,  easily expressed on developing pairs as $\pi^\ast([f,\rho])=[f,\rho|_{\pi_1(\mathcal{Q})}]$ for $\pi_1(\mathcal{Q})<\pi_1(\mathcal{O})$ the subgroup corresponding to the cover.

\begin{proposition}
All Heisenberg structures on orbifolds are complete, and projection onto the holonomy is an embedding $\mathcal{D}_{\Hs^2}(\mathcal{O})\inject \Hom(\pi_1(\mathcal{O}),\Heis)/\Heis_+$.
Under this identification, a finite sheeted covering $\mathcal{Q}\to\mathcal{O}$ describes the deformation space $\mathcal{D}_{\Hs^2}(\mathcal{O})$ 
as the preimage of $\mathcal{D}_{\Hs^2}(\mathcal{Q})$ under the restriction $\pi^\ast\colon \rho\mapsto \rho|_{\pi_1(\mathcal{Q})}$.
\label{Prop:Orbifold_Def}
\end{proposition}
\begin{proof}
Let $\mathcal{O}$ be a Heisenberg orbifold with developing pair $[f,\rho]$, and choose a finite covering $\pi\colon T\to \mathcal{O}$.  
Then by the completeness of $\pi^\ast[f,\rho]\in\mathcal{D}_{\Hs^2}(T)$, the developing map $f$ is a diffeomorphism and $\rho|_{\pi_1(T^2)}$ (hence $\rho$, as $\pi_1(T^2)$ is finite index in $\pi_1(\mathcal{O})$) acts properly discontinuously.
As $\pi_1(T^2)<\pi_1(\mathcal{O})$ is an essential subgroup for all orbifolds covered by the torus, the faithfulness of $\rho|_{\pi_1(T^2)}$ implies faithfulness of $\rho$.
Thus the structure $[f,\rho]$ on $\mathcal{O}$ is complete.
Let $[\phi,\rho]$ be another Heisenberg structure on $\mathcal{O}$ with the same holonomy, then $\phi f\inv:\widetilde{\mathcal{O}}\to\widetilde{\mathcal{O}}$ is $\pi_1(\mathcal{O})$ equivariant and descends to a Heisenberg map $\mathcal{O}\to\mathcal{O}$, inducing the identity on fundamental groups.  
Thus these structures represent the same point in deformation space so projection onto holonomy is an embedding.
\end{proof}

This further restricts the possible topologies of Heisenberg orbifolds.  In particular, any torsion in the fundamental group is represented faithfully by the holonomy so orbifolds may only have corner reflectors and cone points of order two.
In the appendix, we show that all of these actually admit Heisenberg structures, and calculate their deformation spaces.

\begin{corollary}
If $\mathcal{O}$ is a Heisenberg orbifold, necessarily $\mathcal{O}$ is $T^2$, the Klein bottle $\S^1\widetilde{\times}\S^1$, and the pillowcase $\S^2(2,2,2,2)$ or one of their quotients: the cylinder $\S^1\times I$, the Mobius band $\S^1\widetilde{\times} I$, the square $\D^2(\varnothing;2,2,2,2)$, $\D^2(2,2;\varnothing)$, $\D^2(2; 2,2)$ and $\RP^2(2,2)$, .
\end{corollary}

\section{Collapse \& Regenerations}

Unless otherwise specified, $\mathbb{X}$ denotes any one of the constant curvature geometries $\S^2,\E^2$ or $\H^2$ realized as a subgeometry of $\RP^2$ (see Section 2.4) throughout.
Conjugate models will be denoted $C.\mathbb{X}$ for $C\in\GL(3;\R)$.
Recall a collapsing path $[f_t,\rho_t]$ of $\mathbb{X}$ structures degenerates to a Heisenberg structure if 
there is a path $C_t\in\GL(3;\R)$ with $C_t.[f_t,\rho_t]=[C_t f_t,C_t\rho_tC_t\inv]$ converging in the space of developing pairs to $[f_\infty,\rho_\infty]$ with $f_\infty$ an immersion into the affine patch $\Hs^2=\{[x:y:1]\}$ and $\rho_\infty$ with image in $\Heis$.
We may view these rescaled $\mathbb{X}$ structures as geometric structures modeled on the conjugate subgeometry $C_t.\mathbb{X}$, which converge to a Heisenberg structure as $C_t.\mathbb{X}$ itself converges to $\Hs^2$.
The following proposition, a consequence of \cite{CooperDW14} (or a straightforward calculation of conjugacy limits of Lie algebras) describes which conjugacies of $\mathbb{X}\in\{\S^2,\E^2,\H^2\}$ limit to the Heisenberg plane.

\begin{proposition}
\label{prop:Limits}
Let $\mathbb{X}$ be a projective model of a constant curvature geometry in $\RP^2$, and $C_t\colon [0,\infty)\to \PGL(3;\R)$ be any path of projective transformations.
 After potentially rescaling the matrix representatives and
applying the KBH decomposition (Theorem 4.1 of \cite{CooperDW14}) we write $C_t=K_tD_tH_t$ for $K_t\in \mathsf{O}(3)$, $H_t\in\mathsf{Isom}(\mathbb{X})$ and $D_t=\mathsf{diag}(\lambda_t,\mu_t,1)$ with $\lambda_t\geq \mu_t\geq 1$.  
Then for $\mathbb{X}\in\{\S^2,\H^2\}$, the path of geometries $C_t.\mathbb{X}$ limits to the Heisenberg plane if and only if (i) $K_t$ converges in $\O(3)$ and  (ii) $\lambda_t,\mu_t$ and $\lambda_t/\mu_t$ all diverge to $\infty$.
For $\mathbb{X}=\E^2$, the divergence $\lambda_t/\mu_t\to\infty$ alone is necessary and sufficient for (ii).
\end{proposition}

For convenience, we may without loss of generality restrict our attention to conjugacy limits by diagonal matrices $D_t=\diag(\lambda_t,\mu_t,1)$ with $\lambda_t>\mu_t>1$.
To see this, let $\mathbb{X}\in\{\mathbb{S}^2,\mathbb{E}^2,\mathbb{H}^2\}$ and suppose $C_t$ is any path of projective transformations such that $C_t\mathbb{X}\to\Hs^2$.
 Writing $C_t=K_tD_tH_t$ as above, we note $C_t\mathbb{X}=K_tD_t\mathbb{X}$ for all $t$ as $H_t\in\Isom(\mathbb{X})$, and as $K_t$ converges, we see $K_t^{-1}C_t\mathbb{X}=D_t\mathbb{X}$ is conjugate to the original path, even in the limit.

In this section, we classify which Heisenberg tori arise as rescaled limits of collapsing constant-curvature geometric structures.
As all constant-curvature tori are Euclidean, we consider the natural generalization of \emph{conemanifold structures} on the torus, which exist in both positive and negative curvature.

\subsection{Constant Curvature Cone Tori}

\begin{definition}
An $\mathbb{X}$ cone-surface is a surface $\Sigma$ with a complete path metric that is the metric completion of an $\mathbb{X}$-structure on the complement of a discrete set.	
\end{definition}

An $\mathbb{X}$ cone torus $T$ with cone points $C=\{p_1,\ldots p_n\}$ gives an incomplete $\mathbb{X}$-structure on $T_\star^2=T^2\smallsetminus C$ encoded by a class of developing pairs \cite{CooperHK00}.
The space of all such $\mathbb{X}$ cone tori can be identified with the subset $\mathcal{C}_\mathbb{X}(T^2)\subset\mathcal{D}_\mathbb{X}(T^2_\star)$ with metric completions $T^2$, given the subspace topology under this inclusion.

\begin{definition}
A path $T_t$ of $\mathbb{X}$ cone tori converges projectively if the associated incomplete structures $(f_t,\rho_t)\in\mathcal{D}_{\mathbb{X}}(T^2_\star)$ converge in $\mathcal{D}_{\RP^2}(T_\star^2)$ to a projective structure $(f_\infty,\rho_\infty)$, which can be completed to a projective torus $T$. 
Conversely, we say a Heisenberg torus $T$ \emph{regenerates} to $\mathbb{X}$ structures if there is a sequence of $\mathbb{X}$ cone tori converging to $T$ in $\RP^2$.
\end{definition}
 
In the above definition we always require the limiting projective structure on the torus to be nonsingular and allow only sequences of Riemannian cone tori where the cone point(s) vanish in the limit.
Allowing singularities in the limiting structure require a notion of \emph{real projective cone manifold} which is beyond the scope of this work.

In considering the question of regeneration, we further restrict our attention to sequences containing tori with a single cone point.
Cone tori with a single cone point admit a convenient combinatorial description via marked parallelograms, which provides us substantial control.
A marked $\mathbb{X}$-parallelogram is a convex quadrilateral $Q\subset \mathbb{X}$ with opposing geodesic sides of equal length, equipped with an ordering of the vertices $(v_1,v_2,v_3,v_4)$ proceeding counterclockwise from some initial vertex $v_1$.  
 Such a marked parallelogram is determined by a vertex $v=v_1$,the geodesic lengths of the sides adjacent to $v$ and the angle of incidence at $v$.  
The moduli space $\mathcal{P}(\mathbb{X})$ of marked parallelograms in nonpositive curvature is $\R_+^2\times(0,\pi)$, and $\left(0,\tfrac{\pi}{2\kappa}\right)^2\times (0,\pi)$ in spherical space of radius $\kappa$.
Just as deformation space of Euclidean tori can be identified with isometry classes of marked parallelograms $\mathcal{P}(\mathbb{E}^2)$, so can the deformation spaces of $\mathbb{H}^2$ and $\mathbb{S}^2$ cone structures (with the caveat that in positive curvature we must restrict our interest to sufficiently small cone angle).

\begin{proposition}
The map $\mathsf{Glue}\colon\mathcal{P}(\mathbb{X})\to\mathcal{C}_{\mathbb{X}}(T_\star)$ induced by isometrically identifying opposing sides of $Q\in\mathcal{P}(\mathbb{X})$ is a homeomorphism onto its image.  For $\mathbb{X}=\H^2$ this image is the entire deformation space $\mathcal{C}_\mathbb{X}(T_\star)$.
For $\mathbb{X}=\mathbb{S}^2$, the image contains all cone tori whose marked curves each have length less than $\pi/2$.
\label{prop:Gluing}
\end{proposition}
\begin{proof}
There is a unique orientation preserving isometry sending any oriented line segment in $\mathbb{X}$ to any other of the same length.
Thus marked quadrilateral $Q\subset\mathbb{X}$ determines unique side pairings $A,B\in\Isom_+(\mathbb{X})$ identifying opposing sides.
The quotient is a topologically a torus and inherits an $\mathbb{X}$ structure on the complement of $[v]$.
If $Q'$ is isometric to $Q$ then there is a $g\in\Isom(\mathbb{X})$ with $g.Q=Q'$ so the induced structures are isomorphic and $\mathsf{Glue}$ is well defined.

We may also define an inverse \emph{cutting map} as follows.
An marked $\mathbb{X}$ cone torus $T$ has generators $a,b\in\pi_1(T)$ based at the cone point, which
may be pulled tight relative $p$ to length minimizing representatives $\alpha,\beta$ as $T$ is a compact path metric space.
These are locally length minimizing, and so $\mathbb{X}$-geodesics away from $p$.
As $a\simeq\alpha,b\simeq \beta$ generate $\pi_1(T)$, $\alpha$ and $\beta$ have algebraic intersection number $1$.  
As each is globally length minimizing in its pointed homotopy class, the complement $T\smallsetminus\{\alpha\cup\beta\}$ contains no bigons.  
From this it follows that $\alpha\cap\beta=\{p\}$, and so cutting along $\alpha,\beta$ gives a simply connected surface locally modeled on $\mathbb{X}$, with four geodesic boundary components, opposing pairs of which have equal length.

For $\mathbb{X}=\mathbb{H}^2$ such a surface always embeds in $\mathbb{H}^2$ as a hyperbolic parallelogram, so this process defines a map $\mathsf{Cut}\colon\mathcal{C}_{\mathbb{X}}(T_\star)\mapsto\mathcal{P}(\mathbb{X})$.
When $\mathbb{X}=\mathbb{S}^2$ it is possible that the resulting surface does not embed in $\mathbb{S}^2$ (indeed the area of $Q$ may exceed the area of $\mathbb{S}^2$!).
However, if each of $\alpha,\beta$ has length less than $\pi/2$ then $Q$ certainly embeds in $\mathbb{S}^2$ (in fact it embeds into a hemisphere, and thus into the projective model $\RP^2$ of spherical geometry).
These maps $\{\mathsf{Cut},\mathsf{Glue}\}$ are inverses where their composition is defined, and thus define a pair of homeomorphisms.\end{proof}

\begin{figure}[h!]
\centering
\includegraphics[width=0.5\textwidth]{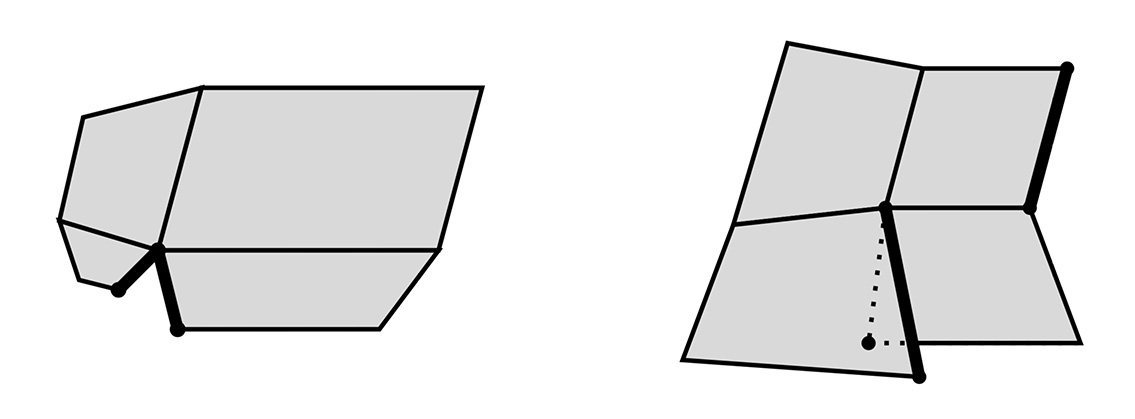}	
\caption{Small portions of the developing map for a hyperbolic and spherical cone torus}
\end{figure}

 To study regenerations from this combinatorial perspective, we characterize when a collapsing path in $\mathcal{C}_\mathbb{X}(T_\star)$ converges in $\mathcal{D}_{\RP^2}(T_\star)$ in terms of marked parallelograms.  First, we show such a characterization is possible as all convergent paths of cone tori admit such a representation.
 
 \begin{proposition}
  Let $\mathbb{X}_t$ be a sequence of geometries conjugate to a constant curvature geometry $\mathbb{X}$ which converge to $\Hs^2$ in the space $\mathfrak{S}_{\RP^2}$ of subgeometries of $\RP^2$.  If $T_t$ is any convergent sequence of $\mathbb{X}_t$ cone tori, then for all sufficiently large $t$ the structures $T_t$   lie in the image of the gluing map $\mathsf{Glue}\colon\mathcal{P}(\mathbb{X})\to \mathcal{C}_\mathbb{X}(T_\star)$. 
  \label{prop:WhenParallelogramsExist}
  \end{proposition}
  \begin{proof}
For $\mathbb{X}=\mathbb{H}^2$, the gluing map is surjective by Proposition \ref{prop:Gluing} so there is nothing more to prove.
For $\mathbb{X}=\mathbb{S}^2$, by the same proposition it is enough to show that eventually all the structures $T_t$ have marked curves of sufficiently short length.
Choose a smooth curve $\gamma$ representing one of the markings on $T_\star$.  For each $t$, we pull $\gamma$ tight fixing the cone point to a geodesic whose length $\ell_t\leq\mathsf{Length}_{T_t}(\gamma)$ defines the length of this marking curve in the $T_t$ structure.
As $t\to\infty$, we show that $\mathsf{Length}_{T_t}(\gamma)$, and hence $\ell_t$ tends to $0$.

Let $(f_t,\rho_t)$ be a convergent sequence of developing pairs for the cone tori $T_t$.
Choosing a lift $\widetilde{\gamma}$ of $\gamma$ we note $\mathsf{Length}_{T_t}(\gamma)=\mathsf{Length}_{\mathbb{X}_t}(f_t\circ\widetilde{\gamma})$, allowing computation of lengths in $T_t$ via the geometry $\mathbb{X}_t$.
 By the assumed convergence of this path of structures, the developing maps $f_t$ converge to some $f\colon \widetilde{T_\star}\to\RP^2$ uniformly in the $C^\infty$ topology on compact sets, so fixing any $\epsilon>0$, for all sufficiently large $t$, 
 $|\mathsf{Length}_{\mathbb{X}_t}(f_t\circ\widetilde{\gamma})-\mathsf{Length}_{\mathbb{X}_t}(f\circ\widetilde{\gamma})|<\ep.$
 
 Thus it suffices to understand the length of the fixed curve $f\circ\widetilde{\gamma}$ in the changing geometries $\mathbb{X}_t$.
 The geometry $\mathbb{X}_t=D_t\mathbb{S}^2$ is a conjugate of spherical geometry by some projective transformation $D_t$, which as in the discussion following Proposition \ref{prop:Limits} we may without loss of generality take to be represented by a diagonal matrix $D_t=\mathsf{diag}(\lambda_t,\mu_t,1)$.
 Changing perspective (applying $D_t^{-1}$) we note
 $\mathsf{Length}_{D_t\mathbb{S}^2}(f\circ\gamma)=\mathsf{Length}_{\mathbb{S}^2}(D_t^{-1}\circ f\circ\widetilde{\gamma})$, allowing us to instead compute the length of a varying curve in a fixed model of $\mathbb{S}^2$.
 
As $D_t\mathbb{S}^2$ limits to the Heisenberg plane, by Proposition \ref{prop:Limits} the eigenvalues $\lambda_t,\mu_t$ of $D_t$ diverge to $\infty$; hence the effect of $D_t^{-1}$ on the standard affine patch $\{[x:y:1]\}$ of $\RP^2$ is to collapse everything towards the origin.
Thus as $t\to\infty$, the curve $D_t^{-1}\circ f\circ\gamma$ converges to a constant map, and its sequence of lengths converges to $0$.
All together this implies for any $\ep>0$, for all sufficiently large $t$,
$$\ell_t\leq \mathsf{Length}_{T_t}(\gamma)=\mathsf{Length}_{\mathbb{X}_t}(f\circ\widetilde{\gamma})<\mathsf{Length}_{\mathbb{X}_t}(f\circ\widetilde{\gamma})+\tfrac{\ep}{2}=\mathsf{Length}_{\mathbb{S}^2}(D_t^{-1}f\circ\widetilde{\gamma})+\tfrac{\ep}{2}<\ep$$
This is stronger than we strictly require: taking $\ep=\pi/2$ and applying this to both marking curves is enough to provide the desired result in light of Proposition \ref{prop:Gluing}.
    \end{proof}
  
  Next, we give a precise description of these convergent sequences of structures in terms of their parallelogram representatives.
 
 \begin{proposition}
 Let $\mathbb{X}_t=D_t\mathbb{X}$ be a sequence of geometries conjugate to $\mathbb{X}$ which converge to $\Hs^2$ in the space $\mathfrak{S}_{\RP^2}$ of subgeometries of $\RP^2$ and $T_t$ a sequence of $\mathbb{X}_t$ cone tori.  Then the structures $T_t$ converge to a Heisenberg torus if and only if:
 \begin{enumerate}
 \item 	For all sufficiently large $t$, there is a choice of embeddings $Q_t\inject\mathbb{X}_t\subset\RP^2$ of the fundamental parallelograms for $T_t$ whose images converge the Hausdorff space of closed subsets of $\RP^2$ to a projective quadrilateral $Q$.
 \item The induced side pairings $A_t,B_t$ of $Q_t$ converge in $\PGL(3;\R)$ to a commuting pair of projective transformations $A,B$.
 \end{enumerate}
 \label{Prop:Quad_Convergence}
 \end{proposition}
\begin{proof}
Again without loss of generality we may assume that the conjugating transformations $D_t$ are represented by diagonal matrices.
Let $(f_t,\rho_t)$ be a convergent sequence of developing pairs for the incomplete structures on $T_\star=T^2\smallsetminus \{\ast\}$ for $\mathbb{X}_t$ cone tori $T_t$.
Choose a generating set $a,b\in\pi_1(T_\star)$ and a basepoint $q\in\widetilde{T_\star}$.
The universal cover $\widetilde{T_\star}$ is tiled by ideal quadrilaterals formed from the lifts of $a,b$.
For each $t$ these can be straightened to geodesics in the  $\mathbb{X}_t$ structure, let $\widetilde{Q_t}\subset\widetilde{T_\star}$ be the geodesic quadrilateral containing $q\in\widetilde{T_\star}$.

By Proposition \ref{prop:WhenParallelogramsExist}, for all sufficiently large $t$, the quadrilateral $\widetilde{Q}_t$ can be embedded as a subset of $\mathbb{X}_t$.
For these structures, the developing map $f_t$ itself provides such an embedding, and we define $Q_t=f_t(\widetilde{Q}_t)=\subset\mathbb{X}_t$ together with the side pairings
$A_t=\rho_t(a)$ and $B_t=\rho_t(b)$.
 When $\mathbb{X}=\mathbb{H}^2$
 The convergence of developing pairs then implies $A_t, B_t$ are convergent in $\PGL(3;\R)$  to $A, B$ and $Q_t$ converges to $Q_\infty$, a fundamental domain for the Heisenberg structure $T$ with sides paired by the commuting transformations $A,B$.

Conversely let $Q_t$ be a sequence of $\mathbb{X}_t$ parallelograms convergent in the Hausdorff space $\mathfrak{C}_{\RP^2}$ of closed subsets of $\RP^2$ to an affine parallelogram $Q$.
The triples $(Q_t, A_t, B_t)$ of the quadrilateral with side pairings define $\mathbb{X}_t$ cone tori, and hence $\RP^2$ punctured tori for all $t$.  
As $t\to \infty$ these converge to a punctured torus $T_\infty$ with holonomy in $\Heis$, and so $T_\infty\in\mathcal{D}_{\Hs^2}(T_\star)$.  
As $[A,B]=I$ the limiting holonomy factors through $\Z\oplus \Z$ and so the limiting torus can be completed to a torus $T_\infty$.
That the limits $A,B\in\Heis$ follows from the definition of $\mathbb{X}_t$ converging to $\Hs^2$, so this limiting projective structure canonically strengthens to a Heisenberg structure.
\end{proof}

\subsection{Translation Tori}

This combinatorial description of cone tori with at most one cone point provides enough control to completely understand the regeneration of translation tori.

\begin{theorem}
Let $\mathbb{X}\in\{\S^2,\E^2,\H^2\}$ and $\mathbb{X}_t=D_t.\mathbb{X}$ be a sequence of diagonal conjugates converging to $\Hs^2$.
Given any translation torus $T$ there is a sequence of $\mathbb{X}_t$ cone tori with at most one cone point converging to $T$.	
\end{theorem}

\begin{proof}[Proof (Euclidean Case):]

Heisenberg tori arise as limits of collapsing families of \emph{smooth} Euclidean tori (there are no Euclidean cone tori with a single cone point, per Gauss-Bonnet).
Let $T$ be a Heisenberg translation torus and $\E_t=D_t.\E^2$ be a sequence of diagonal conjugates of $\E^2$ converging to the Heisenberg plane.  
Choose a fundamental domain $Q$ for $T\subset \Hs^2$, together with side pairings $A,B$ by translations for $T$.
The underlying space for the models $\E^2$, $\E_t$ and $\Hs^2$ in $\RP^2$ are all the entire affine patch $\A^2=\{[x:y:1]\}$; and group $\mathsf{Tr}$ of translations acting on this affine patch is contained in each conjugate $D_t\Isom(\E^2)D_t\inv$ as well as $\Heis$.
Thus $(Q,A,B)$ encodes an $\E_t$-structure $[f,\rho]_{\E_t}$ on $T^2$ for each $t\in\R_+$.
Canonically weakening to projective structures, this is the constant sequence $[f,\rho]_{\RP^2}$ thus clearly convergent.
As $\rho(\Z^2)\subset\mathsf{Tr}< \Heis$, the limit canonically strengthens to the original Heisenberg structure $[f,\rho]_{\Hs^2}$.
\end{proof}

Viewed as Euclidean structures in the fixed model $\E^2$, the developing pairs $[D_t\inv f, D_t\inv \rho D_t]$ encode a collapsing collection of tori with one of the generators of the holonomy shrinking much faster than the other.
That is, even after rescaling to unit area structures this path fails to converge in Teichm\"uller space and limits to a point in the Thurston boundary.
The foliation represented by this point can actually be seen in the limiting Heisenberg structure as the invariant foliation pulled back from $dy$ on $\Hs^2$.

The approach for producing translation tori as limits of hyperbolic and spherical cone tori is similar in spirit, but more involved in the details.
Again we take a fundamental domain with side pairings $(Q,A,B)$ for the proposed limit, and view $Q$ as a geometric parallelogram in each of the model geometries $\mathbb{X}_t$. 
Side pairings $A_t,B_t\in\Isom(\mathbb{X}_t)$ are uniquely determined by each $\mathbb{X}_t$ structure on $Q$, and converge to $A,B$ in the limit.

\begin{figure}[h!]
\centering
\includegraphics[width=0.8\textwidth]{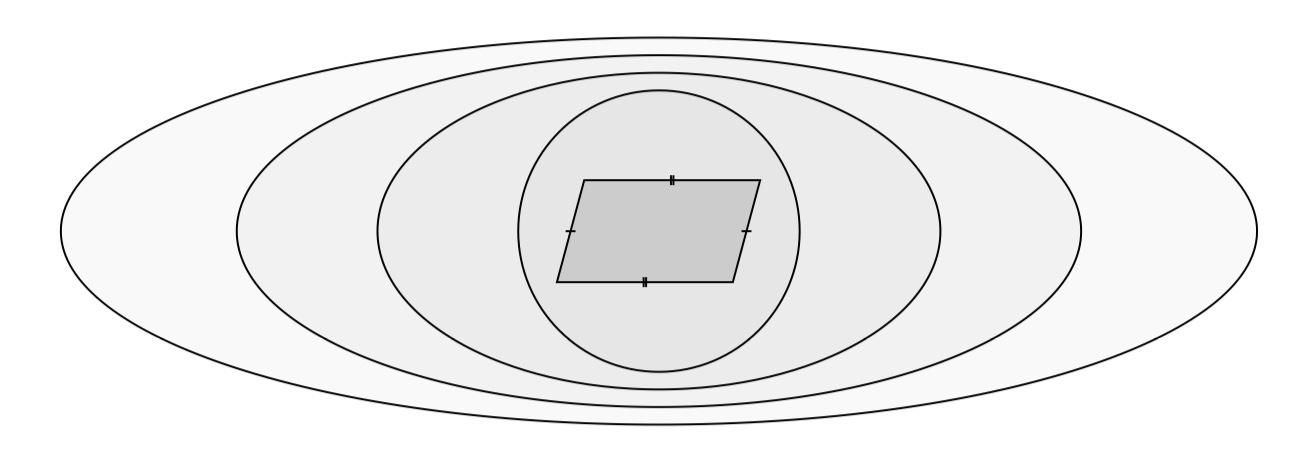}
\vspace{-0.5cm}
\caption{A fixed Quadrilateral and various conjugate models of $\H^2$ containing it.}
\end{figure}
\vspace{-0.5cm}

\begin{proof}[Proof: Hyperbolic and Spherical Cases]

If $\mathbb{X}\in\{\S^2,\H^2\}$, let $Q$ be an origin-centered fundamental domain for $T$ with side pairings $A,B\in\mathsf{Tr}$.  The existence of a convergent sequence of $\mathbb{X}_t$ cone tori $T_t\to T$ follows from the following facts.

\begin{itemize}
\item[] \textbf{Claim 1:} For large $t$, the quadrilateral $Q$ defines an $\mathbb{X}_t$ parallelogram.
\item[] \textbf{Claim 2:} The side pairing $A_t$ preserves the entire projective line through the $\mathbb{X}_t$ midpoints of paired sides.
\item[] \textbf{Claim 3:} If $Q$ is an $\mathbb{X}_t$ parallelogram for all $t$ and $A_t\in\Isom(\mathbb{X}_t)$ pairs opposing sides, $A_t$ converges as a sequence of projective transformations.
\item[] \textbf{Claim 4:} The $\mathbb{X}_t$ midpoints of the edges of $Q$ converge to the Euclidean midpoints as $t\to\infty$.
\end{itemize}

Given that $Q$ defines an $\mathbb{X}_t$ parallelogram, there are unique side pairing transformations $A_t,B_t\in\Isom(\mathbb{X}_t)$ determining an $\mathbb{X}_t$ cone torus.  
By the third claim, these sequences of transformations converge in $\PGL(3,\R)$, and as $\mathbb{X}_t\to\Hs^2$ in fact $A_\infty, B_\infty\in\Heis_0$. 
Recalling the discussion in Section 3, $\Heis_0$ acts simply transitively on the subspace $\mathcal{L}\smallsetminus\mathcal{H}$ of pointed lines, so the limiting transformations are completely determined by their action on  a pair $(p,\ell)$ of a point $p$ on a non-horizontal line $\ell$.  

\begin{figure}[h!]
\centering
\includegraphics[width=0.5\textwidth]{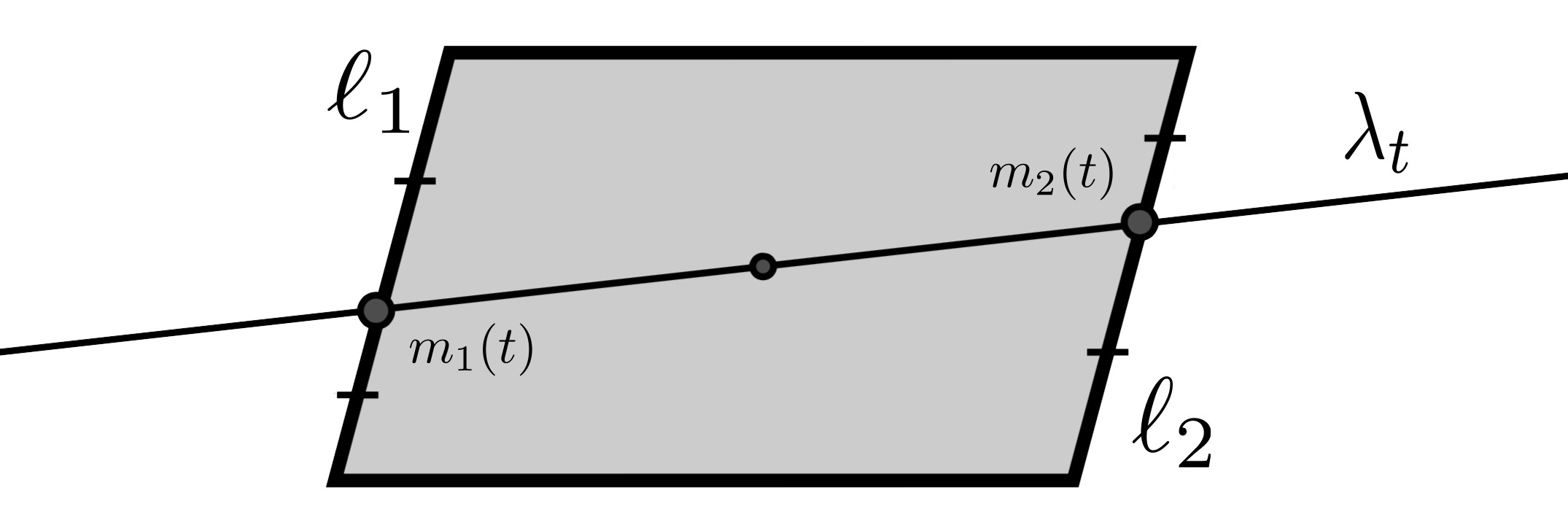}
\vspace{-0.5cm}
\end{figure}

Let $\ell_1,\ell_2$ be a pair of opposing sides of $Q$, with Euclidean midpoints $m_1,m_2$.  
For each $t$, let $m_1(t)$ and $m_2(t)$ be the $\mathbb{X}_t$ corresponding midpoints, and $\lambda_t$ the projective line connecting them.
The second claim implies $A_t$ preserves $\lambda_t$ and so the fourth fact above implies that $A_\infty$ preserves $\lambda=\overline{m_1m_2}$.
Thus $A_\infty$ sends the pair $(m_1,\ell_1)$ to $(m_2,\ell_2)$, as well as the pair $(m_1,\lambda)$ to $(m_2,\lambda)$.  
At least one of the lines $\ell_1,\lambda$ is non-horizontal, and so this completely determines the behavior of $A_\infty$.  
As this agrees precisely with the action of the original transformation $A$, we have $A_\infty=A$ and similarly for $B$.  
Thus the sequence of cone tori corresponding to the triples $(Q, A_t,B_t)$ converge to the original Heisenberg torus $T$ as $t\to\infty$.
\end{proof}

Thus the proof reduces to an argument for the four claims above.  
Throughout its often helpful to switch between the perspectives of a fixed fundamental domain $Q$ in expanding model geometries $\mathbb{X}_t$ and the equivalent picture of shrinking domains $Q_t$ in the fixed model $\mathbb{X}$.

\begin{claim}[1]
Let $Q$ be a affine parallelogram centered at $\vec{0}\in\A^2$ and $\mathbb{X}_t\to\Hs^2$ a sequence of diagonal conjugates of $\mathbb{X}\in\{\S^2,\H^2\}$.  Then for all $t>>0$, $Q$ defines an $\mathbb{X}_t$ parallelogram.	
\end{claim}
\begin{proof}
The $\pi$-rotation about $\vec{0}\in\mathbb{A}^2$ represented by $R=\diag(-1,-1,1)$ is in $\O(3)\cap\O(2,1)$ and is invariant under diagonal conjugacy.  
Thus for each $t$, $R\in\Isom(\mathbb{X}_t)$.  
As $Q$ is an affine parallelogram with centroid $\vec{0}$, $RQ=Q$ so there is an $\mathbb{X}_t$ isometry exchanging opposing sides of $Q$.  
Thus if $Q\subset \mathbb{X}_t$ it defines an $\mathbb{X}_t$ parallelgoram.  
For $\mathbb{X}=\S^2$ this is always satisfied, and for $\mathbb{X}=\H^2$, the domains $\mathbb{X}_t$ limit to the affine patch and so eventually contain any compact subset.
\end{proof}

\begin{claim}[2]
Let $A\in\Isom(\mathbb{X})$ pair opposing sides of the $\mathbb{X}$ parallelogram $Q$.  Then $A$ preserves the projective line through the midpoints of the paired sides.	
\end{claim}
\begin{center}
\includegraphics[width=\textwidth]{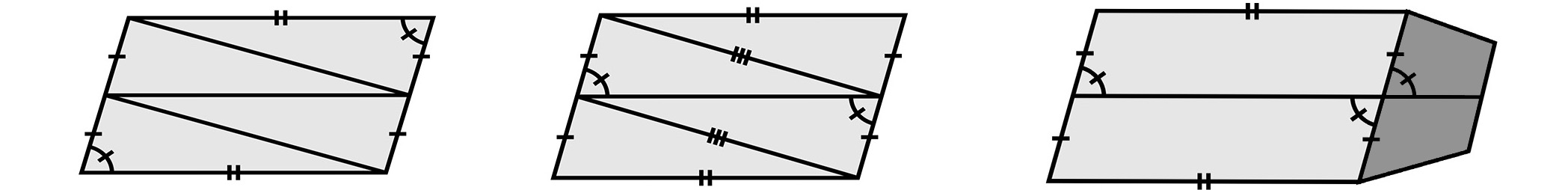}
\end{center}
\vspace{-0.5cm}
\begin{proof}
We argue in classical axiomatic geometry without assuming the parallel postulate as this applies equally to $\S^2,\H^2$.
Opposite angles of a constant-curvature parallelogram are congruent.
Connect the opposing sides of $Q$ paired by $A_t$ with a line segment $\lambda$ through their midpoints. 
This divides $Q$ into two quadrilaterals, subdivided by their diagonals into four triangles.  
The outer two of these triangles are congruent by side-angle-side, and so the diagonals are congruent.  
Thus the inner two triangles are congruent by side-side-side, meaning the opposite angles made by the edges with the line connecting their midpoints are equal.
Consider $Q$ and its translate $A.Q$.
These share an edge, which meets the segments $\lambda$ and $A_t\lambda$ at its midpoint $m$.
As $A$ is an isometry, it follows that opposite angles at $m$ are congruent.
Thus $\lambda$ and $A.\lambda$ are segments of a single projective line, so 
$A$ preserves the line extending $\lambda$ as claimed.
\end{proof}

\begin{claim}[3]
The side pairings $A_t, B_t\in\Isom(\mathbb{X})$ converge in $\PGL(3,\R)$.
\end{claim}
\begin{proof}
A projective transformation of $\RP^2$ is completely determined by its values on a projective basis (a collection of four points in general position). 
The vertices $(v_i)$ of $Q$ form a convenient projective basis with images $(A_tv_i)$ completely specifying the transformations $A_t$.
These transformations converge in $\PGL(3;\R)$ if and only if $(A_tv_i)$ limits to a projective basis, which, as the images $A_tv_i$ remain in a bounded neighborhood of $Q$
\footnote{The conjugating path $C_t$ is \emph{expansive}, with eigenvalues $\lambda_t>\mu_t$ each monotonic in $t$.  Then for $\mathbb{X}=\H^2$, its easy to see $A_tQ\subset AQ$, and for $\mathbb{X}=\S^2$, that $A_tQ<A_0 Q$ for all $t>0$.}
is equivalent to no triangle $\Delta\subset Q$ formed by 3 vertices of $Q$ collapsing in the limit.
That is, it suffices to show 
$\Area_{\E^2}(A_t\Delta)/\Area_{\E^2}(\Delta)\not\to 0$.

Diagonal transformations act linearly on the affine patch and do not change ratios of areas, thus we may transform this to the fixed model $\mathbb{X}$ with a collapsing sequence of triangles $\Delta_t$ being moved by transformations $C_t=D_tA_tD_t\inv$.
For large $t$, both $\Delta_t$ and $C_t\Delta_t$ are extremely close to the origin $\vec{0}\in\A^2$ and we may estimate their area ratio analytically.
By claim 2, $C_t$ preserves the geodesic through the midpoints of paired sides, thus is either a hyperbolic in $\Isom(\H^2)$, or rotation in $\Isom(\S^2)$ with axis represented by an ideal point relative the affine patch.
In each of these cases we may bound the distortion of Euclidean area under these isometries as follows.

Up to conjugation by a rotation we may express any such isometry as
$C=\left(\begin{smallmatrix}c(\tau)&0&s(\tau)\\0&1&0\\ s(\tau)&0&c(\tau)\end{smallmatrix}\right)$, where $(c,s)=(\cosh,\sinh)$ for $\mathbb{X}=\H^2$ and $(\cos,\sin)$ for $\mathbb{X}=\S^2$, and $\tau$ is the translation length along the preserved geodesic.
At any $p=[x:y:1]\in\mathbb{A}^2$ the infinitesimal area distortion $J(p)$ is given by the Jacobian of the projective action of $C$ on the affine patch.
On any region $R\subset\mathbb{A}^2$ then, the overall area distortion $\Area_{\E^2}(C(R))/\Area_{\E^2}{R}$ is bounded below by $J_{\min}(R)=\inf_{p\in R}{J(p)}$ and above by $J_{\max}(R)=\sup_{p\in R}{J(p)}$.

Consider again the region $\Delta_t$ and the side pairing $C_t$ with side pairing of translation length $\tau_t$.  
As $t\to\infty$, both $\Delta_t$ and $C_t\Delta_t$ collapse to $\vec{0}$.
Thus for any $\ep>0$ and all sufficiently large $t$, both of these regions are subsets of the $\ep$-ball about $0$, and 
 we may bound the overall area distortion by $J_{\min}(B_\ep)$ and $J_{\max}(B_\ep)$.
 Computing these, we see
 
$$\frac{1}{(c(\tau) + \ep s(\tau))^3}\leq \frac{\Area_{\E^2}(C_t\Delta_t)}{\Area_{\E^2}(\Delta_t)}\leq\frac{1}{(c(\tau)-\ep s(\tau))^3}.$$

As $t\to\infty$ the translation length $\tau_t$ converges to $0$ (as the geometric structure's developing map collapses to the constant map onto the origin).  Thus the above bounds squeeze the limiting area of $C_t \Delta_t$ to $\Delta_t$ by $1$, so the area of $A_t\Delta$ does not collapse in the limit.
\end{proof}

\begin{claim}[4]
Let $\ell\subset \A^2$ be a line segment and $\mathbb{X}_t\to\Hs^2$ as above.  Then the $\mathbb{X}_t$ midpoint of $\ell$ converges to the Euclidean midpoint.
\end{claim}
\begin{proof}

Let $\ell=\overline{pq}$ and $m\in\ell$ be the Euclidean midpoint.
Viewing $\ell$ in $\mathbb{X}_t$, it has $\mathbb{X}_t$ midpoint $y_t$, and to show $y_t\to m$ it suffices to see $d_{\mathbb{X}_t}(p,m)/d_{\mathbb{X}_t}(m,q)\to 1$. 
Ratios of collinear line segment lengths are invariant under linear transformations, so we may choose to view this situation in the fixed model $\mathbb{X}$ for ease of calculation, with a shrinking line segment $\ell_t=\overline{p_tq_t}$ with Euclidean midpoint $m_t$ and $\mathbb{X}$ midpoint $x_t$.

For $\mathbb{X}=\H^2$ a straightforward computation shows the length of any segment $\ell\subset B_{\E^2}(0,\ep)$ is bounded by a multiple of its Euclidean length 
$\mathsf{Length}_{\E^2}(\ell)\leq \mathsf{Length}_{\mathbb{X}}(\ell)\leq K_\ep \mathsf{Length}_{\E^2}(\ell)$ where $K_\ep$ may be chosen\footnote{For hyperbolic space we may choose $K_\ep=1/\sqrt{1-4\ep^2}$ and for the sphere $K_\ep=1/(1+\ep^2)$ with $\ep$ measured in the Euclidean metric on the affine patch} so that $K_\ep>1, \lim_{\ep\to 0} K_\ep=1$.
Similarly pulling back the spherical metric to the affine patch there is such a $K_\ep>1$ with $\mathsf{Length}_{\E^2}(\ell)/K_\ep \leq \mathsf{Length}_{\mathbb{X}}(\ell)\leq \mathsf{Length}_{\E^2}(\ell)$.
We may use this to bound the difference between the $\mathbb{X}$ and Euclidean midpoints of the shrinking segments $\ell_t$.
	
$$\frac{1}{K_\ep}=\frac{d_{\E^2}(p_t,m_t)}{K_\ep d(m_t,q_t)}\leq \frac{d_{\mathbb{X}}(p_t,m_t)}{d_{\mathbb{X}}(m_t,q_t)}=\frac{d_{\mathbb{X}_t}(p,m)}{d_{\mathbb{X}_t}(m,q)}\leq \frac{K_\ep d_{\E^2}(p_t,m_t)}{d_{\E^2}(m_t,q_t)}=K_\ep.$$
As $\mathbb{X}_t\to\Hs^2$, $\ell_t$ collapses to $\vec{0}$ and we may take smaller and smaller $\ep$ so this ratio converges to 1.
\end{proof}

\subsection{Shear Tori}

Every translation Heisenberg torus arises as a limit of Euclidean, Hyperbolic and Spherical cone tori with at most one cone point.
Translation structures are rather special Heisenberg tori, compromising a codimension-one subset of deformation space.
Here we investigate the generic case, Heisenberg tori with nontrivial shears in their holonomy, and show none regenerate as cone structures with a single cone point.
Shears of the plane fix a single line, and alter the slope of all lines not parallel to this.
All shears in $\Heis$ are parallel, so the holonomy of any shear torus leaves invariant precisely one slope on $\Hs^2$.
This has strong consequences for the distribution of geodesics on Heisenberg orbifolds.

\begin{proposition}
A Heisenberg orbifold $\mathcal{O}$ has a shear in its holonomy if and only if all simple geodesics on $\mathcal{O}$ are pairwise disjoint.
\end{proposition}
\begin{proof}

Let $\mathcal{O}$ be a shear orbifold and $\gamma$ a simple geodesic on $\mathcal{O}$.  
As $\mathcal{O}$ is covered by a complete torus we identify $\widetilde{\mathcal{O}}$ with $\Hs^2$, and the preimage of $\gamma$ under the covering with a $\pi_1(\mathcal{O})$-invariant collection $\{\tilde{\gamma}\}$ of lines in $\Hs^2$.  
As $\gamma$ is simple these are pairwise disjoint and so parallel in $\A^2$.  
Because $\mathcal{O}$ has a shear structure, some $\alpha\in\pi_1(\mathcal{O})$ acts on $\Hs^2$ by a nontrivial shear, which alters the slope of all non-horizontal lines.  Thus, $\{\tilde{\gamma}\}$ is a subset of the horizontal foliation.  
But this holds for any simple geodesic on $\mathcal{O}$ so any two must each lift to a subset of the horizontal foliation, which are then disjoint or (by $\pi_1(\mathcal{O})$ invariance) equal. 
If the two geodesics lift to disjoint collections then their projections are also disjoint, meaning any two distinct simple geodesics on $T$ cannot intersect.

Conversely assume $\mathcal{O}$ is an orbifold covered by a translation torus $T$ given by the developing pair $(f,\rho)$, for $\rho\colon \Z^2\to \mathsf{Tr}$.  
Then $\rho(e_1)$ and $\rho(e_2)$ are linearly independent translations, each preserving each component of a family of parallel lines descending to closed intersecting geodesics on $T$ and further descend to intersecting geodesics on $\mathcal{O}$.
\end{proof}

Hyperbolic, spherical and Euclidean (cone) tori behave quite differently than this.  
Recall that any generators $\langle a,b \rangle=\pi_1(T)$ have geodesic representatives through the cone point and cutting along these gives a constant-curvature parallelogram with side pairings.  
Claim 2 of the previous section shows these side parings must preserve the full projective lines through the midpoints of the paired edges, so these descend to intersecting closed geodesics on $T$.  
The following argument shows this property remains true in the limit.

\begin{theorem}
Let $\mathbb{X}\in\{\S^2,\E^2,\H^2\}$ and $\mathbb{X}_t=D_t \mathbb{X}$ a sequence of conjugate geometries converging to the Heisenberg plane, for $D_t$ diagonal.
Let $T_t$ be a sequence of $\mathbb{X}_t$ cone tori with at most one cone point converging to some Heisenberg torus $T$.
Then $T$ is a translation torus.
\end{theorem}

\begin{proof}

By Proposition \ref{Prop:Quad_Convergence} we may represent these structures by a sequence of $\mathbb{X}_t$ parallelograms $(Q_t,A_t,B_t)$ converging to the triple $(Q_\infty,A_\infty, B_\infty)$ describing the Heisenberg torus $T$.

Claim 2 of the previous section implies that for each $t$, the side pairing $A_t$ preserves the projective line $\alpha_t$ connecting the $\mathbb{X}_t$ midpoints of the paired sides. 
As $t\to\infty$ this sequence of lines in $\RP^2$ subconverges to a projective line $\alpha_\infty$.  
Since $A_t(\alpha_t)=\alpha_t$ for all $t$, it follows that $A_\infty(\alpha_\infty)=\alpha_\infty$, so this line is preserved by the limiting action.
By Claim 3, $\alpha_\infty$ passes through the Euclidean midpoints of opposing sides of $Q_\infty$.
Thus $\alpha_\infty$ and $\beta_\infty$ descend to closed geodesics on $T$. 

As $\alpha_t, \beta_t$ intersect $\partial Q_t$ in the $\mathbb{X}_t$ midpoints of opposing sides, they divide $Q_t$ into four congruent quadrilaterals. Thus the lines $\alpha_t, \beta_t$ intersect at the center of mass of $Q_t$.  
It follows that in the limit the lines $\alpha_\infty,\beta_\infty$ intersect at the center of $Q_\infty$ and  the closed geodesics on $T$ given by the projections of $\alpha_\infty,\beta_\infty$ intersect.
As $T$ has intersecting geodesics, $T$ cannot have any shears in its holonomy, and thus is a translation torus.
\end{proof}

It would be interesting to consider the regeneration of shear tori without restricting to a single cone point.
In particular, whether a sequence of Euclidean cone tori with two cone points, one of cone angle less than $\pi$ and the other greater than $\pi$ could converge to a Heisenberg shear torus provides an intriguing possibility that is yet unknown to the author.
Constructing such examples (or proving the nonexistence thereof) likely requires different techniques than those of section 4.

\section*{Appendix: Heisenberg Orbifolds}

Proposition \ref{Prop:Orbifold_Def} provides a strategy for computing the remaining orbifold deformation spaces:
given $\mathcal{D}(\mathcal{Q})$ and a covering map $\mathcal{Q}\to\mathcal{O}$ we identify $\mathcal{D}(\mathcal{O})$ with the collection of all extensions of $\rho\in\Hom(\pi_1(\mathcal{Q}),\Heis)/\Heis_+$ to $\pi_1(\mathcal{O})$ up to $\Heis$ conjugacy fixing $\rho$.
The following figure shows all Heisenberg orbifolds, with arrows representing the finite covers used in the calculation of their deformation spaces.

\begin{figure}[h!]
\centering
\includegraphics[width=0.8\textwidth]{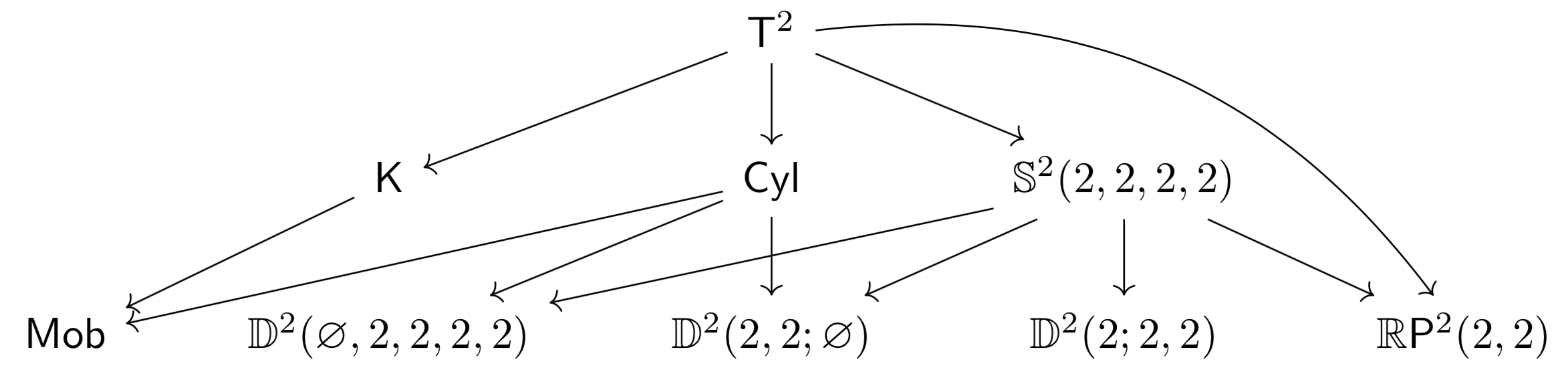}
\caption{All Heisenberg orbifolds are finitely covered by a Heisenberg torus, and furthermore all with cone points or corner reflectors are covered by the pillowcase $\S^2(2,2,2,2)$.}	
\end{figure}

Recall that a \emph{translation torus} has holonomy acting purely by translations.
The Teichm\"uller space of translation tori is homeomorphic to $\R_+\times\S^1$, parameterized by rectangular lattices with ratio of generator lengths in $\R_+$ and angle of first vector $\theta\in\S^1$ with the horizontal.
A translation torus is called \emph{axis aligned} if the holonomy contains a translation along the invariant foliation (up to $\Heis_0$ conjugacy such a structure can actually be assumed to have holonomy generated by translations along the coordinate axes).
Within the Teichm\"uller space $\mathcal{T}_{\Hs^2}(T^2)$, the subset of axis-aligned translation tori is homeomorphic to $\R_+\sqcup\R_+$ corresponding to the points of $\mathcal{F}\cap V(x_1,\; x_2,\; y_1y_2)$.

\begin{proposition}
Every Heisenberg structure on the pillowcase $P=\S^2(2,2,2,2)$ is uniquely covered by a translation torus, and so $\mathcal{T}_{\Hs^2}(P)\cong \R\times\S^1$.
\end{proposition}
\begin{proof}
The twofold branched cover $T\to \S^2(2,2,2,2)=P$ exhibits $\pi_1(P)$ as a $\Z_2=\langle r\rangle$ extension of $\pi_1(T)=\langle a,b\rangle$ with $rar=a\inv$, $rbr=b\inv$.
Thus $\mathcal{D}_{\Hs^2}(P)$ is parameterized by pairs $[\rho, R]$ for $R$ conjugating images under $\rho$ to their inverses.  
Any orientation-preserving element of order two in $\Heis$ is a $\pi$-rotation about some point $p\in\Hs^2$.  Rotations only conjugate translations to their inverses so $\rho$ is the holonomy of a translation torus.  
Given any translation torus, the $\pi$-rotation about any point in the plane provides an extension of $\rho$, and any two are conjugate by conjugacies fixing $\rho$.  Thus restriction provides a bijection from $\mathcal{D}_{\Hs^2}(\S^2(2,2,2,2))$ onto translation tori.
\end{proof}

\begin{proposition}
All Heisenberg Cylinders are quotients of an axis-aligned translation torus, or a shear torus with one generator of the holonomy a horizontal translation.  Thus $\mathcal{T}_{\Hs^2}(\mathsf{Cyl})\cong\R\sqcup \R^2$.
\end{proposition}
\begin{proof}
The doubling mirror double of a cylinder is a torus, and the corresponding orbifold cover $T\to \mathsf{Cyl}$ exhibits $\pi_1(\mathsf{Cyl})$ as a $\Z_2=\langle f\rangle$ extension of $\pi_1(T)$ with $faf=a$, $fbf=b\inv$.  
Thus $\mathcal{D}_{\Hs^2}(\mathsf{Cyl})$ is parameterized by conjugacy classes of pairs $[\rho, F]$ with $\rho\in\mathcal{D}(T)$ and $F$ satisfying the relations above with respect to $\rho(a)$, $\rho(b)$.  
For each $\rho$ with $\rho(a)$ a horizontal translation, there is a one-parameter family of solutions $F$ to the system, all conjugate via conjugacies fixing $\rho$ to a reflection across the horizontal, $\diag\{1,-1,1\}$.
 Thus there is a unique quotient corresponding to each $\rho\in\mathcal{D}_{\Hs^2}(T)$ with $\rho(a)$ a horizontal translation.
If $\rho(a)$ is not a horizontal translation, the system of equations above only has solutions when $\rho\in\mathcal{D}(T)$ is an axis aligned translation torus with $\rho(a)$ vertical, $\rho(b)$ horizontal and $F=\diag\{-1,1,1\}$. 
Thus the Teichm\"uller space consists of the union of the space of axis-aligned tori with all tori having $\rho(a)$ a horizontal translation.	The space of tori with $\rho(a)$ horizontal identifies with a slice $\R_+\times\R$ of $\mathcal{T}_{\Hs^2}(T^2)=\R_+\times\R\times\S^1$ with fixed $\theta=0\in\S^1$, intersecting the space $\R_+\sqcup\R_+$ of axis-aligned translation tori in one copy of $\R_+$. 
\end{proof}

\begin{proposition}
All Heisenberg Klein bottles are quotients of an axis-aligned translation torus, or a shear torus with one generator of the holonomy a horizontal translation.  Thus $\mathcal{T}_{\Hs^2}(\mathsf{K})\cong\R\sqcup \R^2$.
\label{Prop:KleinBottle}
\end{proposition}
\begin{proof}
The Klein bottle $K$ has orientation double cover $T\to K$ corresponding to $\pi_1(K)=\langle x,b\mid xbx\inv=b\inv\rangle$ with $\pi_1(T)=\langle x^2,b\rangle$ so $\mathcal{D}(K)$ is parameterized by pairs $[\rho,X]$ for $\rho\in\mathcal{D}_{\Hs^2}(T)$ and $X^2=\rho(a)$ satisfying $X\rho(b)X\inv\rho(b)=I$.  
As orientation reversing elements of $\Heis$ square to translations, $\rho(a)\in\mathsf{Tr}$, and we distinguish two cases depending on the component $X$ lies in.

If $X\in\diag\{-1,1,1\}\Heis_0$ reflects across the vertical and conjugates $\rho(b)\in\Heis_0$ to its inverse, $\rho(b)$ cannot have any vertical translation component, and so preserves the horizontal foliation.  
As $\rho\in\mathcal{D}_{\Hs^2}(K)$, combining with $\rho(a)\in\mathsf{Tr}$ shows $\rho$ is the holonomy of an axis-aligned translation torus,
and there is a unique solution for $X$ up to conjugacy $\tilde{\rho}(X)=\smat{-1 & 0 & 0\\ 0 &1& r/2\\0&0&1}$.  
If $X\in\diag\{1,-1,1\}\Heis_0$ reflects across the horizontal, the only solutions to $X^2=\rho(a)$ are horizontal translations, and $\rho(b)$ must not have horizontal translational component.
There is a one-parameter family of solutions $X$ to the system, all conjugate via conjugacies fixing $\rho$ to a glide reflection across the horizontal, $\smat{-1 &0&-\lambda/2\\0&1&0\\0&0&1}$.
	
\end{proof}

\begin{corollary}
The space of M\"obius bands identifies with the space of Klein bottles or Cylinders, $\mathcal{T}_{\Hs^2}(\mathsf{M})\cong\R\sqcup \R^2$.
\end{corollary}
\begin{proof}
A Heisenberg M\"obius band has mirror double a Klein bottle and orientation double cover an annulus,	 so points of $\mathcal{D}_{\Hs^2}(M)$ correspond to triples $[\rho, F, X]$ for $[\rho, X]\in\mathcal{D}(K)$, $[\rho, F]\in\mathcal{D}(\mathsf{Cyl})$ satisfying $FX=XF$.  
Every $\rho\in\mathcal{D}_{\Hs^2}(T)$ that extends to a representation of $\pi_1(\mathsf{Cyl})$ does so uniquely, and also uniquely extends to a representation of $\pi_1(K)$ and so there is a unique M\"obius band covered by the torus with holonomy $\rho$.
\end{proof}

\begin{proposition}
Each Heisenberg structure on $\mathcal{O}\in\{D^2(2,2;\varnothing), \D^2(\varnothing,2,2,2,2),\RP^2(2,2)\}$ is the quotient of a unique axis-aligned translation torus.
Thus $\mathcal{T}_{\Hs^2}(\mathcal{O})\cong \R_+\sqcup\R_+$.
\end{proposition}
\begin{proof}
These three orbifolds are twofold covered by $\S^2(2,2,2,2)$, and thus fourfold covered by translation tori.
The orbifolds $\D^2(2,2;\varnothing)$ and $\D^2(\varnothing;2,2,2,2)$ are also covered by the annulus, and the only translation annuli are axis aligned.  
Each such axis aligned torus has a unique $\D^2(2,2;\varnothing)$ and $\D^2(\varnothing;2,2,2,2)$ quotient.  
The orbifold $\RP^2(2,2)$ arises as a fourfold quotient of the torus by glide reflections $x,y$ such that $\pi_1(T^2)=\langle x^2,y^2\rangle$.  
As seen in the Proposition \ref{Prop:KleinBottle}, each glide reflection squaring to a generator of $\pi_1(T^2)$ is along an axis of $\R^2$, so in this case the torus cover must be an axis-aligned translation torus.  
Each such admits a unique $\RP^2(2,2)$ quotient.

\end{proof}

\begin{proposition}
The orbifold $\D^2(2;2,2)$ has Teichm\"uller space homeomorphic to $\R\sqcup\R$.
\end{proposition}
\begin{proof}
This orbifold is the quotient of the pillowcase by a reflection passing through two opposing cone points, and thus is fourfold covered by a translation torus.
Algebraically this is an extension of $\pi_1(P)=\langle a,b,r\rangle$ by $\langle f\rangle=\Z_2$ satisfying $faf=b$, $fbf=a$, $frf=r\inv$.
Up to $\Heis_+$ conjugacy we may choose representations for homothety classes of translation tori translating along $v_\theta=\smat{\cos \theta\\\sin\theta}$ and $\lambda v_\theta^\perp=\smat{-\lambda \sin\theta\\\lambda\cos\theta}$ uniquely defined for $\theta\in[0,\pi)$, $\lambda>0$.
The only reflections $F$ representing $f$ are parallel to the $x$ or $y$ axes; so the covering torus $T$ cannot be axis 
aligned for this to pass through the cone points of the pillow quotient.
For $F\in\diag(-1,1,1)\Heis_0$ computing with the relations shows there is a solution if and only if $\theta\in(0,\pi)$ and $\lambda=\tan\theta$.
Similarly, for $F\in\diag(1,-1,1)\Heis_0$, a solution exists for $\theta\in(\pi/2,\pi)$ and $\lambda=-\tan\theta$.
These solutions are unique up to conjugacy and so $\mathsf{T}_{\Hs^2}(\D^2(2;2,2))\cong\R\sqcup\R$.
\end{proof}

%
%
%

\bibliography{Refs}
\bibliographystyle{amsalpha}

\end{document}